\documentclass{amsart}
\usepackage{amsmath}
\usepackage{amssymb}
\usepackage{amsthm}
\usepackage{epsfig}
\usepackage{amsfonts}
\usepackage{amscd}
\usepackage{mathrsfs}
\usepackage{enumerate}
\usepackage{latexsym}
\usepackage{url}

\newtheorem{theorem}{Theorem}[section]
\newtheorem{lemma}[theorem]{Lemma}
\newtheorem{proposition}[theorem]{Proposition}

\theoremstyle{definition}
\newtheorem{definition}[theorem]{Definition}
\newtheorem{example}[theorem]{Example}

\newtheorem{remark}[theorem]{Remark}

\theoremstyle{question}
\newtheorem{question}[theorem]{Question}

\numberwithin{equation}{section}

\newtheorem*{xclaim}{Claim}

\begin{document}

\title[Topological extensions with compact remainder]{Topological extensions with compact remainder}

\author{M.R. Koushesh}
\address{Department of Mathematical Sciences, Isfahan University of Technology, Isfahan 84156--83111, Iran.}
\address{School of Mathematics, Institute for Research in Fundamental Sciences (IPM), P.O. Box: 19395--5746, Tehran, Iran.}
\email{koushesh@cc.iut.ac.ir}
\thanks{This research was in part supported by a grant from IPM (No. 90030052).}

\subjclass[2010]{Primary 54D35, 54D60; Secondary 54D40.}

\keywords{Stone--\v{C}ech compactification, Hewitt realcompactification, pseudocompactification, realcompactification, Mr\'{o}wka's condition $(\mathrm{W})$, compactness-like topological property}

\begin{abstract}
Let $\mathfrak{P}$ be a topological property. We study the relation between the order structure of the set of all $\mathfrak{P}$-extensions of a completely regular space $X$ with compact remainder (partially ordered by the standard partial order $\leq$) and the topology of certain subspaces of the outgrowth $\beta X\setminus X$. The cases when $\mathfrak{P}$ is either pseudocompactness or realcompactness are studied in more detail.
\end{abstract}

\maketitle

\section{Introduction}

\emph{All spaces under consideration are assumed to be completely regular; completely regular spaces are Hausdorff.}

A space $Y$ is called an {\em extension} of a space $X$ if $Y$ contains $X$ as a dense subspace. If $Y$ is an extension of $X$ then the subspace $Y\setminus X$ of $Y$ is called the {\em remainder} of $Y$. Two extensions of $X$ are said to be {\em equivalent} if there exists a homeomorphism between them which fixes $X$ point-wise. This defines an equivalence relation on the class of all extensions of $X$. The equivalence classes will be identified with individuals when no confusion arises. For any extensions $Y$ and $Y'$ of $X$ we let $Y\leq Y'$ if there exists a continuous mapping of $Y'$ to $Y$ which fixes $X$ point-wise. The relation $\leq$ defines a partial order on the set of all (equivalence classes of) extensions of $X$. (See Section 4.1 of \cite{PW} for more details.) Let $\mathfrak{P}$ be a topological property. An extension $Y$ of $X$ is called a {\em  $\mathfrak{P}$-extension} if it has  $\mathfrak{P}$. If $\mathfrak{P}$ is compactness, pseudocompactness or realcompactness, respectively, then the corresponding $\mathfrak{P}$-extensions are called {\em compactification}, {\em pseudocompactification} or {\em realcompactification}. (Recall that a space is said to be {\em pseudocompact} if every continuous real-valued mapping defined on it is bounded. Also, a space is called {\em realcompact} if it is homeomorphic to a closed subspace of a product of copies of the real line.) For any partially ordered sets $(P,\leq)$ and $(Q,\leq)$, a mapping $f:P\rightarrow Q$ is called an {\em order-homomorphism} ({\em anti-order-homomorphism}, respectively), if $f(a)\leq f(b)$ ($f(b)\leq f(a)$, respectively) whenever $a\leq b$. An order-homomorphism (anti-order-homomorphism, respectively) $f:P\rightarrow Q$ is called an {\em order-isomorphism} ({\em anti-order-isomorphism}, respectively), if $f^{-1}:Q\rightarrow P$ (exists and) is an order-homomorphism (anti-order-homomorphism, respectively). Two partially ordered sets $(P,\leq)$ and $(Q,\leq)$ are said to be {\em order-isomorphic} ({\em anti-order-isomorphic}, respectively), if there exists an  order-isomorphism (anti-order-isomorphism, respectively) between them. A {\em zero-set} of a space $X$ is a set of the form $\mathrm{Z}(f)=f^{-1}(0)$, where $f:X\rightarrow[0,1]$ is continuous. Any set of the form $X\setminus Z$, where $Z$ is a zero-set of $X$, is called a {\em cozero-set} of $X$. We denote the set of all zero-sets of $X$ by ${\mathscr Z}(X)$, and the set of all cozero-sets of $X$ by $\mathrm{Coz}(X)$.

There is a well-known result of K.D. Magill, Jr. which for a locally compact space $X$ relates the order-structure of the set ${\mathscr K}(X)$ of all compactifications of $X$ to the topology of the space $\beta X\setminus X$ (where $\beta X$ is the Stone--\v{C}ech compactification of $X$).

\begin{theorem}[Magill \cite{Mag3}]\label{KLFA}
Let $X$ and $Y$ be locally compact non-compact spaces. If ${\mathscr K}(X)$ and ${\mathscr K}(Y)$ are order-isomorphic then $\beta X\setminus X$ and $\beta Y\setminus Y$ are homeomorphic.
\end{theorem}

There have been extensive studies to generalize Magill's theorem in various directions. (See \cite{Me} for a different proof of the theorem; see \cite{R} for generalizations of the theorem to non-locally compact spaces; see \cite{Wo1} and \cite{Do} for a zero-dimensional version of the theorem, and see \cite{PW1} for an  extension of the theorem to mappings.) One of the most interesting generalizations is the one obtained by J. Mack, M. Rayburn and R.G. Woods in \cite{MRW1}. We state some results from \cite{MRW1} below.

Let $X$ be a space and let $\mathfrak{P}$ be a topological property. A $\mathfrak{P}$-extension of $X$ is called {\em tight} if it does not contain properly any other $\mathfrak{P}$-extension of $X$. Suppose that $\mathfrak{P}$ is closed hereditary, productive and is such that if a space is the union of a compact space and a space with $\mathfrak{P}$ then it has $\mathfrak{P}$. Define the {\em $\mathfrak{P}$-reflection} $\gamma_\mathfrak{P} X$ of $X$ by
\[\gamma_\mathfrak{P}X=\bigcap\{T:T \mbox{ has }\mathfrak{P}\mbox{ and } X\subseteq T\subseteq\beta X\}.\]
If $\mathfrak{P}$ is compactness then $\gamma_\mathfrak{P} X=\beta X$ and if $\mathfrak{P}$ is realcompactness then $\gamma_\mathfrak{P}X=\upsilon X$, where $\upsilon X$ is the Hewitt realcompactification of $X$. (Recall that for any space $X$ the {\em Hewitt realcompactification} of $X$ is the space $\upsilon X$ which contains $X$ as a dense subspace and is such that every continuous $f:X\rightarrow\mathbb{R}$ is continuously extendible to $\upsilon X$; one may assume that $\upsilon X\subseteq\beta X$.) Denote by ${\mathscr P}(X)$ the set of all tight $\mathfrak{P}$-extensions of $X$. For a non-$\mathfrak{P}$ locally-$\mathfrak{P}$ space (that is, a space in which every point has a neighborhood with $\mathfrak{P}$) $X$ there is the largest one-point extension $X^*$ in ${\mathscr P}(X)$. Let
\[{\mathscr P}^*(X)=\big\{T\in{\mathscr P}(X):X^*\leq T\big\}\]
and for any $T\in{\mathscr P}^*(X)$, if $f_T:\beta X\rightarrow\beta T$ denotes the continuous extension of $\mathrm{id}_X$, let
\[{\mathscr D}^*(X)=\big\{T\in {\mathscr P}^*(X):f_T[\gamma_\mathfrak{P} X]=T\big\}.\]

The following main result of \cite{MRW1} generalizes Magill's theorem.

\begin{theorem}[Mack, Rayburn and  Woods \cite{MRW1}]\label{GGF}
Let $X$ and $Y$ be locally-$\mathfrak{P}$ non-$\mathfrak{P}$ spaces and suppose that ${\mathscr D}^*(X)={\mathscr P}^*(X)$ and ${\mathscr D}^*(Y)={\mathscr P}^*(Y)$. If $\gamma_\mathfrak{P}X\setminus X$ and $\gamma_\mathfrak{P}Y\setminus Y$ are $C^*$-embedded in
$\gamma_\mathfrak{P}X$ and $\gamma_\mathfrak{P}Y$, respectively, then ${\mathscr P}^*(X)$ and ${\mathscr P}^*(Y)$ are lattice-isomorphic if and only if $\gamma_\mathfrak{P}X\setminus X$ and $\gamma_\mathfrak{P}Y\setminus Y$ are homeomorphic.
\end{theorem}

In particular, when $\mathfrak{P}$ is realcompactness, Theorem \ref{GGF} gives the following.

\begin{theorem}[Mack, Rayburn and  Woods \cite{MRW1}]\label{FGGF}
Let $X$ and $Y$ be locally realcompact spaces such that $\upsilon X$ and $\upsilon Y$ are both Lindel\"{o}f. Then ${\mathscr P}^*(X)$ and ${\mathscr P}^*(Y)$ are lattice-isomorphic if and only if $\upsilon X\setminus X$ and $\upsilon Y\setminus Y$ are homeomorphic.
\end{theorem}

Motivated by the above results and our previous studies \cite{Ko1}--\cite{Ko6} (see also \cite{Ko8} and \cite{Ko7}) we prove the following analogous results. For a space $X$ we denote by ${\mathscr U}(X)$ the set of all pseudocompactifications of $X$ with compact remainder.

\begin{theorem}[Theorems \ref{RSQ} and \ref{RTRT}]\label{OOS}
Let $X$ and $Y$ be locally pseudocompact non-pseudocompact spaces. If ${\mathscr U}(X)$ and ${\mathscr U}(Y)$ are order-isomorphic then
\begin{itemize}
\item[\rm(1)] $\mathrm{cl}_{\beta X}(\beta X\setminus\upsilon X)$ and $\mathrm{cl}_{\beta Y}(\beta Y\setminus\upsilon Y)$ are homeomorphic.
\item[\rm(2)] If in addition $X$ and $Y$ are locally compact, then $(\beta X\setminus X)\setminus\mathrm{cl}_{\beta X}(\beta X\setminus\upsilon X)$ and $(\beta Y\setminus Y)\setminus\mathrm{cl}_{\beta Y}(\beta Y\setminus\upsilon Y)$ are homeomorphic.
\end{itemize}
\end{theorem}

Analogously, if for a space $X$ we denote by ${\mathscr R}(X)$ the set of all realcompactifications of $X$ with compact remainder, we prove the following.

\begin{theorem}[Theorems \ref{AEF} and \ref{PKS}]\label{AFK}
Let $X$ and $Y$ be locally realcompact non-realcompact spaces. If ${\mathscr R}(X)$ and ${\mathscr R}(Y)$ are order-isomorphic then
\begin{itemize}
\item[\rm(1)] $\mathrm{cl}_{\beta X}(\upsilon X\setminus X)$ and $\mathrm{cl}_{\beta Y}(\upsilon Y\setminus Y)$ are homeomorphic.
\item[\rm(2)] If in addition $X$ and $Y$ are locally compact, then $(\beta X\setminus X)\setminus\mathrm{cl}_{\beta X}(\upsilon X\setminus X)$ and $(\beta Y\setminus Y)\setminus\mathrm{cl}_{\beta Y}(\upsilon Y\setminus Y)$ are homeomorphic.
\end{itemize}
\end{theorem}

We further extend the above two theorems by considering $\mathfrak{P}$-extensions with compact remainder. Let $X$ be a space and let $\mathfrak{P}$ be a topological property. Define
\[\lambda_\mathfrak{P} X=\bigcup\big\{\mathrm{int}_{\beta X}\mathrm{cl}_{\beta X}C:C\in\mathrm{Coz}(X)\mbox{ and } \mathrm{cl}_X C \mbox{ has }\mathfrak{P}\big\}.\]
If $\mathfrak{P}$ is pseudocompactness then
\[\lambda_\mathfrak{P} X=\mathrm{int}_{\beta X}\upsilon X\]
and if $\mathfrak{P}$ is realcompactness (and $X$ is normal) then
\[\lambda_\mathfrak{P} X=\beta X\setminus\mathrm{cl}_{\beta X}(\upsilon X\setminus X).\]
Denote by ${\mathscr E}_\mathfrak{P}(X)$ the set of all $\mathfrak{P}$-extensions of $X$ with compact remainder. As in \cite{Ko3}, we call $\mathfrak{P}$ a {\em compactness-like} topological property if $\mathfrak{P}$ is hereditary with respect to clopen (simultaneously closed and open) subspaces, is both invariant and inverse invariant under perfect surjective mappings (recall that a closed continuous mapping $f:X\rightarrow Y$ is {\em perfect}, if each fiber $f^{-1}(y)$, where $y\in Y$, is a compact subspace of $X$) and satisfies Mr\'{o}wka's condition $(\mathrm{W})$ (that is, if a space $Y$ contains a point $p$ with an open base ${\mathscr B}$ for $Y$ at $p$ such that $Y\setminus B$ has $\mathfrak{P}$ for each $B\in{\mathscr B}$, then $Y$ has $\mathfrak{P}$). Neither  pseudocompactness nor realcompactness is a compactness-like topological property. We prove the following.

\begin{theorem}[Theorems \ref{GFO} and \ref{PPKS}]\label{GFDD}
Let $X$ and $Y$ be locally-$\mathfrak{P}$ non-$\mathfrak{P}$ spaces, where $\mathfrak{P}$ is a compactness-like topological property. If ${\mathscr E}_\mathfrak{P}(X)$ and ${\mathscr E}_\mathfrak{P}(Y)$ are order-isomorphic then
\begin{itemize}
\item[\rm(1)] $\beta X\setminus\lambda_\mathfrak{P}X$ and $\beta Y\setminus\lambda_\mathfrak{P}Y$ are homeomorphic.
\item[\rm(2)] If in addition $X$ and $Y$ are locally compact, then $\lambda_\mathfrak{P}X\setminus X$ and $\lambda_\mathfrak{P}Y\setminus Y$ are homeomorphic.
\end{itemize}
\end{theorem}

We now briefly review some known facts from the theory of the Stone--\v{C}ech compactification. Additional information may be found in \cite{E} and \cite{GJ}.

\subsection*{The Stone--\v{C}ech compactification.} The {\em Stone--\v{C}ech compactification} $\beta X$ of a space $X$ is the largest (with respect to the partial order $\leq$) compactification of $X$ and is characterized among all compactifications of $X$ by either of the following properties:
\begin{itemize}
  \item Every continuous $f:X\rightarrow K$, where $K$ is a compact space, is continuously extendable to $\beta X$; denote by $f_\beta$ this continuous extension of $f$.
  \item For every $Z,S\in {\mathscr Z}(X)$ we have
  \[\mathrm{cl}_{\beta X}(Z\cap S)=\mathrm{cl}_{\beta X}Z\cap\mathrm{cl}_{\beta X}S.\]
\end{itemize}
In what follows use will be made of the following properties of $\beta X$.
\begin{itemize}
  \item $X$ is locally compact if and only if $X$ is open in $\beta X$.
  \item Any clopen subspace of $X$ has clopen closure in $\beta X$.
  \item If $X\subseteq T\subseteq\beta X$ then $\beta T=\beta X$.
  \item If $X$ is normal then $\beta T=\mathrm{cl}_{\beta X}T$ for any closed subspace $T$ of $X$.
\end{itemize}

\subsection*{The Hewitt realcompactification.} The {\em Hewitt realcompactification} $\upsilon X$ of a space $X$ is the largest (with respect to the partial order $\leq$) realcompactification of $X$ and is characterized among all realcompactifications of $X$ by the following property:
\begin{itemize}
  \item Every continuous $f:X\rightarrow Y$, where $Y$ is a realcompact space, is continuously extendable to $\upsilon X$.
\end{itemize}
The Hewitt realcompactification of $X$ may be viewed as the intersection of all cozero-sets of $\beta X$ which contain $X$. Thus, the points of $\beta X\setminus\upsilon X$ are exactly those $p\in\beta X$ for which there exists a $G_\delta$-set of $\beta X$ containing $p$ and missing $X$.

\section{Pseudocompactifications with compact remainder}

Pseudocompact extensions are called {\em pseudocompactifications}. (Recall that a space is said to be {\em pseudocompact} if every continuous real-valued mapping  defined on it is bounded.) In this section we consider pseudocompactifications with compact remainder. The section is divided into two parts. The first part consists of some known results which describe the general form of all pseudocompactifications of a given space $X$ with compact remainder. The second part deals with the partially ordered set of all pseudocompactifications of a space $X$ with compact remainder. We show that this partially ordered set determines the topology of certain subspaces of $\beta X\setminus X$.

\subsection{Pseudocompactifications with compact remainder; their general form.}\label{IUI} The results of this part are from \cite{Ko5} (for a proof of Lemma \ref{j2}, see \cite{Ko3}); we include them here for completeness of  results and reader's convenience.

\begin{definition}\label{TDAS}
For a space $X$ denote by ${\mathscr K}(X)$ and ${\mathscr U}(X)$ the set of all compactifications of $X$ and the set of all pseudocompactifications of $X$ with compact remainder, respectively.
\end{definition}

\begin{lemma}\label{j2}
Let $X$ be a space, let $Y$ be an extension of $X$ with compact remainder and let $\phi:\beta X\rightarrow\beta Y$ continuously extend $\mathrm{id}_X$. Then $\beta Y$ coincides with the quotient space of $\beta X$ obtained by contracting each fiber $\phi^{-1}(p)$, where $p\in Y\setminus X$, to $p$, and $\phi$ is the quotient mapping.
\end{lemma}

\begin{lemma}\label{ASD}
Let $X$ be a space, let $Y$ be an extension of $X$ with compact remainder, let $\zeta Y$ be a compactification of $Y$ and let $\phi:\beta X\rightarrow\zeta Y$ continuously extend $\mathrm{id}_X$. The following are equivalent:
\begin{itemize}
\item[\rm(1)] $Y\in{\mathscr U}(X)$.
\item[\rm(2)] $\mathrm{cl}_{\beta X}(\beta X\setminus\upsilon X)\subseteq\phi^{-1}[Y\setminus X]$.
\end{itemize}
\end{lemma}

\begin{proof}
First consider the case when $\zeta Y=\beta Y$. Note that since $\phi^{-1}[Y\setminus X]$ is closed in $\beta X$, condition (2) is equivalent to the requirement that $\beta X\setminus\upsilon X\subseteq\phi^{-1}[Y\setminus X]$.

(1) {\em implies} (2). Let $x\in\beta X\setminus\upsilon X$. Suppose to the contrary that $x\notin\phi^{-1}[Y\setminus X]$. Let $P\in{\mathscr Z}(\beta X)$ be such that $x\in P$ and $P\cap X$ is empty. Now $G=P\setminus\phi^{-1}[Y\setminus X]$ is non-empty (as it contains $x$) and it is a countable intersection of open subspaces  of $\beta X$ each missing $\phi^{-1}[Y\setminus X]$. Thus (using Lemma \ref{j2}) $G$ is a non-empty $G_\delta$-set of $\beta Y$ which misses $Y$, contradicting the pseudocompactness of $Y$.

(2) {\em implies} (1). Suppose to the contrary that $Y$ is not pseudocompact. Let $p\in \beta Y\setminus\upsilon Y$ and let $Z\in {\mathscr Z}(\beta Y)$ be such that $p\in Z$ and $Z\cap Y$ is empty. Then $\phi^{-1}[Z]\in{\mathscr Z}(\beta X)$ misses $X$, and thus
\[\phi^{-1}[Z]\subseteq\beta X\setminus\upsilon X\subseteq\phi^{-1}[Y\setminus X].\]
Since $p\in\phi^{-1}[Z]$ (as $\phi(p)=p$; see Lemma \ref{j2}) we have $p\in\phi^{-1}[Y\setminus X]$ or equivalently $p=\phi(p)\in Y\setminus X$, which contradicts the choice of $Z$.

Suppose that $\zeta Y$ is an arbitrary compactification of $Y$. Denote by $\psi:\beta X\rightarrow\zeta Y$ and $\gamma:\beta Y\rightarrow\zeta Y$ the continuous extensions of $\mathrm{id}_X$ and $\mathrm{id}_Y$, respectively. Then $\gamma\psi=\phi$, as they agree on $X$, and $\gamma[\beta Y\setminus Y]=\zeta Y\setminus Y$. The lemma now follows, as
\[\psi^{-1}[Y\setminus X]=\psi^{-1}\big[\gamma^{-1}[Y\setminus X]\big]=(\gamma\psi)^{-1}[Y\setminus X]=\phi^{-1}[Y\setminus X].\]
\end{proof}

\begin{definition}\label{LFS}
A space $X$ is called {\em locally pseudocompact} if every $x\in X$ has an open neighborhood $U$ in $X$ with pseudocompact closure $\mathrm{cl}_X U$.
\end{definition}

Note that pseudocompactness is hereditary with respect to regular closed subspaces; thus, a space with a pseudocompactification with compact remainder is locally pseudocompact.

The following lemma may be used in the sequel without explicit reference.

\begin{lemma}[Comfort \cite{C}]\label{FGA}
A space $X$ is locally pseudocompact if and only if $X\subseteq\mathrm{int}_{\beta X}\upsilon X$.
\end{lemma}

\begin{theorem}\label{RES}
Let $X$ be a locally pseudocompact space, let $\zeta X$ be a compactification of $X$, let $\phi:\beta X\rightarrow\zeta X$ continuously extend $\mathrm{id}_X$ and let $E$ be a compact subspace of $\zeta X\setminus X$ containing $\phi[\mathrm{cl}_{\beta X}(\beta X\setminus\upsilon X)]$. Then $Y=X\cup E\in{\mathscr U}(X)$ (considered as a subspace of $\zeta X$). Furthermore, every element of ${\mathscr U}(X)$ is of this form.
\end{theorem}

\begin{proof}
This follows from Lemma \ref{ASD} (and Lemma \ref{FGA}).
\end{proof}

\subsection{Pseudocompactifications with compact remainder; their partially ordered sets.}\label{TIS} In this part we prove the first set of main results.

In our first theorem, for a locally pseudocompact space $X$, we relate the order-structure of the set ${\mathscr U}(X)$ to the topology of the subspace $\mathrm{cl}_{\beta X}(\beta X\setminus\upsilon X)$ of the outgrowth. Here is the proof overview. We define an order-isomorphism $\Theta_X$ from the set ${\mathscr K}(\mathrm{int}_{\beta X}\upsilon X)$ of all compactifications of $\mathrm{int}_{\beta X}\upsilon X$ to ${\mathscr U}(X)$. We then characterize order-theoretically the image of $\Theta_X$ in ${\mathscr U}(X)$; for this purpose we need to consider certain types of co-atoms of ${\mathscr U}(X)$. Thus, for any locally pseudocompact spaces $X$ and $Y$, any order-isomorphism between ${\mathscr U}(X)$ and ${\mathscr U}(Y)$ carries the image of $\Theta_X$ onto the image of $\Theta_Y$, and therefore, induces an order-isomorphism between ${\mathscr K}(\mathrm{int}_{\beta X}\upsilon X)$ and ${\mathscr K}(\mathrm{int}_{\beta Y}\upsilon Y)$. Magill's theorem (Theorem \ref{KLFA}) will then imply that $\beta X\setminus\mathrm{int}_{\beta X}\upsilon X$ and $\beta Y\setminus\mathrm{int}_{\beta Y}\upsilon Y$ are homeomorphic. Now, we proceed with the proof details (in a possibly different order).

\begin{definition}
Let $X$ be a space and let $Y$ be an extension of $X$. Let $Z$ be a compactification of $Y$ and let $\phi:\beta X\rightarrow Z$ be the continuous extension of $\mathrm{id}_X$. Define
\[{\mathscr F}_X(Y)=\big\{\phi^{-1}(p):p\in Y\setminus X\big\}.\]
We may write ${\mathscr F}(Y)$ instead of ${\mathscr F}_X(Y)$ when no confusion arises. Note that the definition is independent of the choice of the compactification $Z$. To see this, let $\psi:\beta X\rightarrow\beta Y$ and $\gamma:\beta Y\rightarrow Z$ denote the continuous extensions of $\mathrm{id}_X$ and $\mathrm{id}_Y$, respectively. Then $\gamma\psi=\phi$, as they coincide on $X$. Also $\gamma[\beta Y\setminus Y]=Z\setminus Y$. Thus $\gamma^{-1}(p)=p$ for each $p\in Y\setminus X$, and
\begin{eqnarray*}
\big\{\phi^{-1}(p):p\in Y\setminus X\big\}&=&\big\{(\gamma\psi)^{-1}(p):p\in Y\setminus X\big\}\\&=&\big\{\psi^{-1}\big(\gamma^{-1}(p)\big):p\in Y\setminus X\big\}=\big\{\psi^{-1}(p):p\in Y\setminus X\big\}.
\end{eqnarray*}
\end{definition}

The following lemma is known (see \cite{Ko3}); we use it very often, mostly without referring to it.

\begin{lemma}\label{DFH}
Let $X$ be a space and let $Y_1$ and $Y_2$ be extensions of $X$ with compact remainder. The following are equivalent:
\begin{itemize}
\item[\rm(1)] $Y_1\leq Y_2$.
\item[\rm(2)] Each element of ${\mathscr F}(Y_2)$ is contained in an element of ${\mathscr F}(Y_1)$.
\end{itemize}
\end{lemma}

\begin{definition}\label{IOHJD}
Let $(X,\leq)$ be a partially ordered set with the largest element $u$. An element $a\in X$ is called a {\em co-atom} in $X$ if $a\neq u$ and there exists no $x\in X$ with $a<x<u$.
\end{definition}

Co-atoms in ${\mathscr U}(X)$ play a crucial role here; but first, we need to know that ${\mathscr U}(X)$ has a largest element. This is shown in \cite{Ko5}, however, it can be readily deduced at this point.

\begin{definition}\label{TR}
For a space $X$ let
\[\zeta X=X\cup\mathrm{cl}_{\beta X}(\beta X\setminus\upsilon X)=X\cup(\beta X\setminus\mathrm{int}_{\beta X}\upsilon X)\]
considered as a subspace of $\beta X$.
\end{definition}

\begin{lemma}\label{PFAJ}
Let $X$ be a locally pseudocompact space. Then ${\mathscr U}(X)$ has a largest element, namely $\zeta X$.
\end{lemma}

\begin{proof}
Note that $X\subseteq\mathrm{int}_{\beta X}\upsilon X$, by Lemma \ref{FGA}. Therefore
\[\zeta X\setminus X=\beta X\setminus\mathrm{int}_{\beta X}\upsilon X\]
is compact. The lemma now follows from Lemmas \ref{ASD} and \ref{DFH}; note that
\[{\mathscr F}(\zeta X)=\big\{\{y\}:y\in\mathrm{cl}_{\beta X}(\beta X\setminus\upsilon X)\big\}.\]
\end{proof}

Next, we identify the co-atoms of ${\mathscr U}(X)$.

\begin{definition}\label{HGFD}
Let $X$ be a locally pseudocompact space and let $C_1,\dots,C_n$ be $n$ pairwise disjoint compact non-empty subspaces of $\beta X\setminus X$. Let $Z$ be the quotient space of $\beta X$ obtained by contracting $C_1,\dots,C_n$ to $p_1,\dots,p_n$, respectively, with the quotient mapping $q:\beta X\rightarrow Z$. Define
\[e_X(C_1,\dots,C_n)=X\cup\{p_1,\dots,p_n\}\cup\big(\mathrm{cl}_{\beta X}(\beta X\setminus\upsilon X)\setminus (C_1\cup\dots\cup C_n)\big),\]
considered as a subspace of $Z$. Note that $Z$ is a compactification of $Y$, thus
\[e_X(C_1,\dots,C_n)\in{\mathscr U}(X)\]
by Lemma \ref{ASD} (with $\phi=q$ in its statement), and
\begin{eqnarray*}
&&{\mathscr F}\big(e_X(C_1,\dots,C_n)\big)=\\&&\big\{q^{-1}(p):p\in e_X(C_1,\dots,C_n)\setminus X\big\}=\\&&\{C_1,\dots,C_n\}\cup\big\{\{y\}:y\in\mathrm{cl}_{\beta X}(\beta X\setminus\upsilon X)\setminus (C_1\cup\dots\cup C_n)\big\}.
\end{eqnarray*}
\end{definition}

The following characterizes the co-atoms of ${\mathscr U}(X)$; it may be used in the sequel without explicit reference.

\begin{lemma}\label{UGGH}
Let $X$ be a locally pseudocompact space and let $Y\in{\mathscr U}(X)$. Then $Y$ is a co-atom in ${\mathscr U}(X)$ if and only if $Y$ is of either of the following forms.
\begin{itemize}
\item[\rm(1)] $Y=e_X(\{a\})$, for some $a\in(\beta X\setminus X)\setminus\mathrm{cl}_{\beta X}(\beta X\setminus\upsilon X)$.
\item[\rm(2)] $Y=e_X(\{a,b\})$, for some distinct $a,b\in\mathrm{cl}_{\beta X}(\beta X\setminus\upsilon X)$.
\end{itemize}
\end{lemma}

\begin{proof}
Let $Y$ be a co-atom in ${\mathscr U}(X)$. There exist no distinct $F,G\in{\mathscr F}(Y)$ with $|F|>1$ and $|G|>1$; as otherwise $Y<e_X(F)<\zeta X$. Then either
\begin{description}
\item[{\sc Case 1}] $|F|=1$ for each $F\in{\mathscr F}(Y)$, or else
\item[{\sc Case 2}] $|F|>1$ for exactly one $F\in{\mathscr F}(Y)$.
\end{description}
In the first case, by Lemma \ref{ASD} we have
\[\big\{\{y\}:y\in\mathrm{cl}_{\beta X}(\beta X\setminus\upsilon X)\big\}\subsetneqq{\mathscr F}(Y),\]
where the proper inclusion is because $Y\neq\zeta X$. Now there exists at least one
\[a\in(\beta X\setminus X)\setminus\mathrm{cl}_{\beta X}(\beta X\setminus\upsilon X)\]
with $\{a\}\in{\mathscr F}(Y)$, and there exists at most one such $a$; as if $\{b\}\in{\mathscr F}(Y)$, where
\[b\in(\beta X\setminus X)\setminus\mathrm{cl}_{\beta X}(\beta X\setminus\upsilon X)\]
is distinct from $a$, then $Y<e_X(\{b\})<\zeta X$. This shows that in this case $Y$ is of the form indicated in (1). In the second case, we have $|F|=2$; otherwise, choose some distinct $a,b\in F$ and note that $Y<e_X(\{a,b\})<\zeta X$. Finally, note that $F\subseteq\mathrm{cl}_{\beta X}(\beta X\setminus\upsilon X)$, as if $a\notin\mathrm{cl}_{\beta X}(\beta X\setminus\upsilon X)$ for some $a\in F$, then $Y<e_X(\{a\})<\zeta X$. Thus, in this case $Y$ is of the form indicated in (2).

To prove the converse, let $Y$ be as indicated in (1). Then
\[{\mathscr F}(Y)=\big\{\{a\}\big\}\cup\big\{\{y\}:y\in\mathrm{cl}_{\beta X}(\beta X\setminus\upsilon X)\big\}\]
and for each $T\in{\mathscr U}(X)$ with $Y\leq T\leq\zeta X$, depending on whether $\{a\}\in{\mathscr F}(T)$ or not, we have $T=Y$ or $T=\zeta X$.  That is, $Y$ is a co-atom in ${\mathscr U}(X)$. Now, let $Y$ be as indicated in (2). Then
\[{\mathscr F}(Y)=\big\{\{a,b\}\big\}\cup\big\{\{y\}:y\in\mathrm{cl}_{\beta X}(\beta X\setminus\upsilon X)\setminus\{a,b\}\big\}.\]
Let $T\in{\mathscr U}(X)$ with $Y\leq T\leq\zeta X$ and $G\in{\mathscr F}(T)$. Then $G\subseteq F$ for some $F\in{\mathscr F}(Y)$. If $F\neq\{a,b\}$, then $G$ is a singleton, and if $F=\{a,b\}$, then either
\[G=\{a\},G=\{b\},\mbox{ or }G=\{a,b\}.\]
In the first two cases $T=\zeta X$, and in the latter case $T=Y$. That is, $Y$ is a co-atom in ${\mathscr U}(X)$.
\end{proof}

\begin{definition}\label{JFS}
Let $X$ be a locally pseudocompact space. We say that a co-atom $Y$ of ${\mathscr U}(X)$ is {\em of type $(\mathrm{I})$}, if
$Y=e_X(\{a\})$, where
\[a\in(\beta X\setminus X)\setminus\mathrm{cl}_{\beta X}(\beta X\setminus\upsilon X);\]
otherwise, $Y$ is said to be {\em of type $(\mathrm{II})$}.
\end{definition}

We now define an order-isomorphism of ${\mathscr K}(\mathrm{int}_{\beta X}\upsilon X)$ to ${\mathscr U}(X)$ and characterize its image order-theoretically.

\begin{lemma}\label{UGRH}
Let $X$ be a locally pseudocompact space. Define
\[\Theta:{\mathscr K}(\mathrm{int}_{\beta X}\upsilon X)\rightarrow{\mathscr U}(X)\]
by
\[\Theta(T)=X\cup(T\setminus\mathrm{int}_{\beta X}\upsilon X)\]
for any $T\in{\mathscr K}(\mathrm{int}_{\beta X}\upsilon X)$. Then $\Theta$ is an order-isomorphism onto its image.
\end{lemma}

\begin{proof}
To show that $\Theta$ is well defined, let $T\in{\mathscr K}(\mathrm{int}_{\beta X}\upsilon X)$. Then $\Theta(T)$ is an extension of $X$, and the remainder \[\Theta(T)\setminus X=T\setminus\mathrm{int}_{\beta X}\upsilon X\]
is compact; as $\mathrm{int}_{\beta X}\upsilon X$, being open in $\beta X$, is locally compact, and thus, open in all of its compactifications. Let $\phi:\beta X\rightarrow T$ be continuous and fix the points of $\mathrm{int}_{\beta X}\upsilon X$. (Observe that $\beta(\mathrm{int}_{\beta X}\upsilon X)=\beta X$, as $X\subseteq\mathrm{int}_{\beta X}\upsilon X\subseteq\beta X$.) Note that $T$ is a compactification of $\Theta(T)$ (as $\Theta(T)$ contains $X$, and $X$, being dense in $\mathrm{int}_{\beta X}\upsilon X$, is dense in $T$) and
\[\phi[\beta X\setminus\mathrm{int}_{\beta X}\upsilon X]=T\setminus\mathrm{int}_{\beta X}\upsilon X.\]
Now since
\[\phi^{-1}\big[\Theta(T)\setminus X\big]=\phi^{-1}[T\setminus\mathrm{int}_{\beta X}\upsilon X]=\phi^{-1}\big[\phi[\beta X\setminus\mathrm{int}_{\beta X}\upsilon X]\big]\supseteq\beta X\setminus\mathrm{int}_{\beta X}\upsilon X,\]
by Lemma \ref{ASD} it follows that $\Theta(T)$ is pseudocompact.

Now, we show that $\Theta$ is an order-homomorphism. Suppose that $S\leq T$ for some $S,T\in{\mathscr K}(\mathrm{int}_{\beta X}\upsilon X)$. Then there exists a continuous $f:T\rightarrow S$, fixing $\mathrm{int}_{\beta X}\upsilon X$ point-wise. Note that
\[f[T\setminus\mathrm{int}_{\beta X}\upsilon X]=S\setminus\mathrm{int}_{\beta X}\upsilon X,\]
which yields
\begin{eqnarray*}
\Theta(S)&=&X\cup(S\setminus\mathrm{int}_{\beta X}\upsilon X)\\&=&f[X]\cup f[T\setminus\mathrm{int}_{\beta X}\upsilon X]=f\big[X\cup (T\setminus\mathrm{int}_{\beta X}\upsilon X)\big]=f\big[\Theta(T)\big]
\end{eqnarray*}
and therefore
\[g=f|\Theta(T):\Theta(T)\rightarrow\Theta(S).\]
Since $g$ fixes the points of $X$, this shows that $\Theta(S)\leq\Theta(T)$.

Before we proceed with the remainder of the proof we need to verify the following.

\begin{xclaim}
If $T\in{\mathscr K}(\mathrm{int}_{\beta X}\upsilon X)$ then $\beta(\Theta(T))=T$.
\end{xclaim}

\noindent \emph{Proof of the claim.}
Let $\phi:\beta(\Theta(T))\rightarrow T$ be continuous and fix the points of $\Theta(T)$. To prove the claim it suffices to show that $\phi$ is a homeomorphism.   But (since $\phi$ is surjective with compact domain) $\phi$ is a homeomorphism, if it is injective, and since it fixes the points of $\Theta(T)$ and
\[\phi\big[\beta\big(\Theta(T)\big)\setminus\Theta(T)\big]=T\setminus\Theta(T),\]
the mapping $\phi$ is injective provided that it is injective on $\beta(\Theta(T))\setminus\Theta(T)$. Let $a,b\in\beta(\Theta(T))\setminus\Theta(T)$ with $\phi(a)=\phi(b)$. Note that $X$, being dense in $T$, is dense in $\Theta(T)$, and thus in $\beta(\Theta(T))$, that is,  $\beta(\Theta(T))$ is also a compactification of $X$. Denote by $\psi:\beta X\rightarrow \beta(\Theta(T))$ the continuous extension of $\mathrm{id}_X$. Recall that $\beta(\Theta(T))$ is the quotient space of $\beta X$ obtained by contracting each $\psi^{-1}(p)$, where $p\in\Theta(T)\setminus X$, to a point with the quotient mapping $\psi$. (See Lemma \ref{j2}.) Since $\Theta(T)\in{\mathscr U}(X)$ we have
\[\beta X\setminus\mathrm{int}_{\beta X}\upsilon X\subseteq\psi^{-1}\big[\Theta(T)\setminus X\big],\]
by Lemma \ref{ASD}. Thus $a,b\in\mathrm{int}_{\beta X}\upsilon X\setminus X$, as $a,b\notin\Theta(T)\setminus X$, and also $\psi(a)=a$ and $\psi(b)=b$. Note that $\phi\psi:\beta X\rightarrow T$ fixes the points of $\mathrm{int}_{\beta X}\upsilon X$, as it fixes the points of its dense subspace $X$. We have
\[a=\phi\big(\psi(a)\big)=\phi(a)=\phi(b)=\phi\big(\psi(b)\big)=b\]
which proves the claim.

\medskip

Suppose that $\Theta(S)\leq\Theta(T)$ for some $S,T\in{\mathscr K}(\mathrm{int}_{\beta X}\upsilon X)$. We show that $S\leq T$. Let $h:\Theta(T)\rightarrow \Theta(S)$ be continuous and fix the points of $X$. Using the claim, $h$ can be continuously extended to a mapping
$h_\beta:T\rightarrow S$. Note that $h_\beta$ fixes the points of $\mathrm{int}_{\beta X}\upsilon X$, as it fixes the points of its dense subspace $X$, and therefore $S\leq T$. This in particular implies that $\Theta$ is injective and
\[\Theta^{-1}:\mathrm{Im}(\Theta)\rightarrow{\mathscr K}(\mathrm{int}_{\beta X}\upsilon X)\]
is an order-homomorphism.
\end{proof}

\begin{definition}\label{KHF}
For a locally pseudocompact space $X$ denote by
\[\Theta_X:{\mathscr K}(\mathrm{int}_{\beta X}\upsilon X)\rightarrow{\mathscr U}(X)\]
the order-isomorphism (onto its image) defined for any $T\in{\mathscr K}(\mathrm{int}_{\beta X}\upsilon X)$ by
\[\Theta_X(T)=X\cup(T\setminus\mathrm{int}_{\beta X}\upsilon X).\]
\end{definition}

\begin{lemma}\label{GHDS}
Let $X$ be a locally pseudocompact space and let $C$ and $D$ be compact non-empty subspaces of $\beta X\setminus X$. Then
\begin{itemize}
\item[\rm(1)] $e_X(C)\wedge e_X(D)=e_X(C\cup D)$, if $C\cap D\neq\emptyset$.
\item[\rm(2)] $e_X(C)\wedge e_X(D)=e_X(C,D)$, if $C\cap D=\emptyset$.
\end{itemize}
\end{lemma}

\begin{proof}
(1). By Lemma \ref{DFH} it is clear that
\[e_X(C\cup D)\leq e_X(C)\mbox{ and }e_X(C\cup D)\leq e_X(D).\]
Let
\[T\leq e_X(C)\mbox{ and }T\leq e_X(D)\]
for some $T\in{\mathscr U}(X)$. Then (again by Lemma \ref{DFH}) we have $C\subseteq F$ and $D\subseteq G$ for some $F,G\in{\mathscr F}(T)$. Now $F\cap G$ is non-empty, as $C\cap D$ is so, therefore $F=G$. Thus $C\cup D$ is contained in an element of ${\mathscr F}(T)$, showing that $T\leq e_X(C\cup D)$. The proof for part (2) is analogous.
\end{proof}

Next, we characterize order-theoretically those co-atoms of ${\mathscr U}(X)$ which are of type $(\mathrm{II})$ (and thus those which are of type $(\mathrm{I})$ as well).

\begin{lemma}\label{GFS}
Let $X$ be a locally pseudocompact non-pseudocompact space and let $T$ be a co-atom in ${\mathscr U}(X)$. The following are equivalent:
\begin{itemize}
\item[\rm(1)] $T$ is of type $(\mathrm{II})$.
\item[\rm(2)] There exists a co-atom $S$ in ${\mathscr U}(X)$ with
\[\big|\big\{U\in{\mathscr U}(X):U\geq S\wedge T\big\}\big|=5.\]
\end{itemize}
\end{lemma}

\begin{proof}
(1) {\em  implies} (2). Let $T=e_X(\{a,b\})$ where $a,b\in\mathrm{cl}_{\beta X}(\beta X\setminus\upsilon X)$ are distinct. Choose some $c\in\mathrm{cl}_{\beta X}(\beta X\setminus\upsilon X)$ distinct from both $a$ and $b$; this is possible since $X$ is non-pseudocompact.  (Indeed
\[|\beta X\setminus\upsilon X|\geq2^{2^{\aleph_0}};\]
see Problem 5Z of \cite{PW}.) Let $S=e_X(\{b,c\})$. Then $S$ is a co-atom in ${\mathscr U}(X)$ and by Lemma \ref{GHDS} we have
\[S\wedge T=e_X\big(\{b,c\}\big)\wedge e_X\big(\{a,b\}\big)=e_X\big(\{a,b,c\}\big).\]
Let $U\in{\mathscr U}(X)$ be such that $U\geq S\wedge T$. Let $F\in{\mathscr F}(U)$. Then $F$ is contained in an element of ${\mathscr F}(S\wedge T)$, by Lemma \ref{DFH}. If $F\cap\{a,b,c\}$ is empty then $F$ is a singleton. If $F\cap\{a,b,c\}$ is non-empty then $F\subseteq\{a,b,c\}$, and thus $U$ is either
\[e_X\big(\{a,b,c\}\big),e_X\big(\{a,b\}\big),e_X\big(\{a,c\}\big),e_X\big(\{b,c\}\big)\mbox{ or }\zeta X.\]
Conversely, if $U$ is either of the above elements then $U\geq S\wedge T$. That the above elements are all distinct follows from Lemma \ref{DFH}.

(2) {\em  implies} (1). Suppose that $T$ is of type $(\mathrm{I})$. Then $T=e_X(\{a\})$, where
\[a\in(\beta X\setminus X)\setminus\mathrm{cl}_{\beta X}(\beta X\setminus\upsilon X).\]
Now let $S$ be a co-atom in ${\mathscr U}(X)$. We have the following cases.
\begin{description}
\item[{\sc Case 1}] Suppose that $S$ is of type $(\mathrm{I})$. Then $S=e_X(\{b\})$ where
\[b\in(\beta X\setminus X)\setminus\mathrm{cl}_{\beta X}(\beta X\setminus\upsilon X).\]
If $a=b$ then $S\wedge T=e_X(\{a\})$ by Lemma \ref{GHDS}, and thus, since $S\wedge T$ is a co-atom in ${\mathscr U}(X)$, there exist only 2 elements $U\in{\mathscr U}(X)$ with $U\geq S\wedge T$, namely, $\zeta X$ and $e_X(\{a\})$ itself. If $a\neq b$, then
\[S\wedge T=e_X\big(\{a\},\{b\}\big)\]
by Lemma \ref{GHDS}, and thus for each $U\in{\mathscr U}(X)$ with $U\geq S\wedge T$, depending on whether $\{a\}\in{\mathscr F}(U)$ or $\{b\}\in{\mathscr F}(U)$ (or neither), $U$ equals to either
    \[e_X\big(\{a\},\{b\}\big),e_X\big(\{a\}\big),e_X\big(\{b\}\big)\mbox{ or }\zeta X.\]
\item[{\sc Case 2}] Suppose that $S$ is of type $(\mathrm{II})$. Then $S=e_X(\{b,c\})$ for some distinct $b,c\in\mathrm{cl}_{\beta X}(\beta X\setminus\upsilon X)$. We have
    \[S\wedge T=e_X\big(\{a\},\{b,c\}\big)\]
    by Lemma \ref{GHDS}, and thus, the elements $U\in{\mathscr U}(X)$ with $U\geq S\wedge T$ are exactly
    \[e_X\big(\{a\},\{b,c\}\big),e_X\big(\{a\}\big),e_X\big(\{b,c\}\big)\mbox{ and }\zeta X.\]
\end{description}
Therefore, in either case there exist at most 4 elements $U\in{\mathscr U}(X)$ with $U\geq S\wedge T$.
\end{proof}

The next result, together with Lemma \ref{GFS}, gives an order-theoretic characterization of the elements of $\mathrm{Im}(\Theta_X)$.

\begin{lemma}\label{TDT}
Let $X$ be a locally pseudocompact space and let $T\in{\mathscr U}(X)$. The following are equivalent:
\begin{itemize}
\item[\rm(1)] $T\in\mathrm{Im}(\Theta_X)$.
\item[\rm(2)] There exists no co-atom $S$ in ${\mathscr U}(X)$ of type $(\mathrm{I})$ with $S\geq T$.
\end{itemize}
\end{lemma}

\begin{proof}
(1) {\em  implies} (2). Suppose that $T=\Theta_X(U)$ for some $U\in{\mathscr K}(\mathrm{int}_{\beta X}\upsilon X)$. Let $\phi:\beta X\rightarrow U$ be continuous and fix the points of $\mathrm{int}_{\beta X}\upsilon X$. (Note that $\beta(\mathrm{int}_{\beta X}\upsilon X)=\beta X$.) Then (since $U$ is a compactification of $T$)
\[{\mathscr F}_X(T)=\big\{\phi^{-1}(p):p\in T\setminus X\big\}.\]
But
\[T\setminus X=U\setminus\mathrm{int}_{\beta X}\upsilon X\]
(by the definition of $\Theta_X$), and therefore (since
\[\phi[\beta X\setminus\mathrm{int}_{\beta X}\upsilon X]=U\setminus\mathrm{int}_{\beta X}\upsilon X,\]
as $\phi$ fixes the points of $\mathrm{int}_{\beta X}\upsilon X$) we have $F\subseteq\beta X\setminus\mathrm{int}_{\beta X}\upsilon X$ for each $F\in{\mathscr F}_X(T)$. By Lemma \ref{DFH} this implies that any co-atom $S$ in ${\mathscr U}(X)$ with $S\geq T$ is of type $(\mathrm{II})$.

(2) {\em  implies} (1).  Let $\psi:\beta X\rightarrow\beta T$ be the continuous extension of $\mathrm{id}_X$. Then
\begin{equation}\label{YF}
\mathrm{cl}_{\beta X}(\beta X\setminus\upsilon X)\subseteq\psi^{-1}[T\setminus X],
\end{equation}
by Lemma \ref{ASD}. Suppose that the inclusion in (\ref{YF}) is proper and let
\[a\in\psi^{-1}[T\setminus X]\setminus\mathrm{cl}_{\beta X}(\beta X\setminus\upsilon X).\]
Then $e_X(\{a\})$ is a co-atom in ${\mathscr U}(X)$ of type $(\mathrm{I})$ with $e_X(\{a\})\geq T$. This contradiction proves the equality in (\ref{YF}). Note that $\beta T$ is the quotient space of $\beta X$ obtained by contracting each $\psi^{-1}(p)$, where $p\in T\setminus X$, to a point, with the quotient mapping $\psi$. (See Lemma \ref{j2}.) Therefore
\[\beta T=\mathrm{int}_{\beta X}\upsilon X\cup(T\setminus X),\]
and thus $\beta T$ is a compactification of $\mathrm{int}_{\beta X}\upsilon X$. (Note that $X$, being dense in $T$, is dense in $\beta T$, and $X\subseteq\mathrm{int}_{\beta X}\upsilon X$.) We have
\[T=X\cup(T\setminus X)=X\cup(\beta T\setminus\mathrm{int}_{\beta X}\upsilon X)=\Theta_X(\beta T)\in\mathrm{Im}(\Theta_X).\]
\end{proof}

We are now ready to prove our first main result.

\begin{theorem}\label{RSQ}
Let $X$ and $Y$ be locally pseudocompact non-pseudocompact spaces. If ${\mathscr U}(X)$ and ${\mathscr U}(Y)$ are order-isomorphic then $\mathrm{cl}_{\beta X}(\beta X\setminus\upsilon X)$ and $\mathrm{cl}_{\beta Y}(\beta Y\setminus\upsilon Y)$ are homeomorphic.
\end{theorem}

\begin{proof}
By Lemmas \ref{GFS} and \ref{TDT}, if ${\mathscr U}(X)$ is order-isomorphic to ${\mathscr U}(Y)$, then $\mathrm{Im}(\Theta_X)$ is order-isomorphic to $\mathrm{Im}(\Theta_Y)$, and thus ${\mathscr K}(\mathrm{int}_{\beta X}\upsilon X)$ is order-isomorphic to ${\mathscr K}(\mathrm{int}_{\beta Y}\upsilon Y)$, by Lemma \ref{UGRH}. Theorem \ref{KLFA} now implies that $\beta X\setminus\mathrm{int}_{\beta X}\upsilon X$ and $\beta Y\setminus\mathrm{int}_{\beta Y}\upsilon Y$ are homeomorphic.
\end{proof}

Below, we show that the converse of Theorem \ref{RSQ} is not true in general (while the converse of Magill's theorem, Theorem  \ref{KLFA}, is indeed true; see \cite{Mag3}).

\begin{example}\label{HFP}
Let $X=D(\aleph_0)$ (the discrete space of cardinality $\aleph_0$) and let $Y=D(\aleph_0)\oplus[0,\Omega)$ (where $\Omega$ is the first uncountable ordinal and $\oplus$ denotes the disjoint union). Then
\[\beta Y=\beta\big(D(\aleph_0)\big)\oplus[0,\Omega]\mbox{ and }\upsilon Y=D(\aleph_0)\oplus[0,\Omega].\]
Thus
\[\mathrm{cl}_{\beta Y}(\beta Y\setminus\upsilon Y)=\mathrm{cl}_{\beta X}(\beta X\setminus\upsilon X).\]
However, ${\mathscr U}(X)$ and ${\mathscr U}(Y)$ are not order-isomorphic, as ${\mathscr U}(Y)$ contains a co-atom of type $(\mathrm{I})$ (since $\Omega\notin\mathrm{cl}_{\beta Y}(\beta Y\setminus\upsilon Y)$), while ${\mathscr U}(X)$ does not.
\end{example}

The following question naturally arises in connection with Theorem \ref{RSQ}.

\begin{question}\label{HLTD}
For a space $X$ let ${\mathscr A}(X)$ denote the set of all pseudocompactifications of $X$. For a (locally pseudocompact) space $X$, does the order structure of ${\mathscr A}(X)$ (partially ordered by $\leq$) determine the topology of $\mathrm{cl}_{\beta X}(\beta X\setminus\upsilon X)$?
\end{question}

In our next theorem, for a locally compact space $X$, we relate the order-structure of the set ${\mathscr U}(X)$ to the topology of the subspace $(\beta X\setminus X)\setminus\mathrm{cl}_{\beta X}(\beta X\setminus\upsilon X)$ of the outgrowth. Here is the proof overview. We order-theoretically characterize the elements of the set of all one-point extensions of $X$ contained in ${\mathscr U}(X)$ (denoted by ${\mathscr U}^*(X)$); for this purpose we need to introduce certain types of elements of ${\mathscr U}(X)$ (called {\em co-atom covers}). Thus, for any locally pseudocompact spaces $X$ and $Y$, any order-isomorphism between ${\mathscr U}(X)$ and ${\mathscr U}(Y)$ carries ${\mathscr U}^*(X)$ onto ${\mathscr U}^*(Y)$, and therefore, induces an order-isomorphism between ${\mathscr U}^*(X)$ and ${\mathscr U}^*(Y)$. Now it is a known result (see \cite{Ko4}) that any order-isomorphism between ${\mathscr U}^*(X)$ and ${\mathscr U}^*(Y)$ induces an order-isomorphism between the set of all closed subspaces of $(\beta X\setminus X)\setminus\mathrm{cl}_{\beta X}(\beta X\setminus\upsilon X)$ and $(\beta Y\setminus Y)\setminus\mathrm{cl}_{\beta Y}(\beta Y\setminus\upsilon Y)$, partially ordered by $\subseteq$. Since the topology of any space is determined by the order-structure of the set of all of its closed subspaces (see Theorem 11.1 of \cite{B}), this will then prove our result. Now, we proceed with the proof details.

\begin{definition}\label{OGFD}
Let $(X,\leq)$ be a partially ordered set with the largest element $u$. An element $d\in X$ is called a {\em co-atom cover} in $X$ if there exists exactly one $x\in X$ with $d<x<u$. If $d$ is a co-atom cover in $X$, denote by $d'$ the unique element $x\in X$ such that $d<x<u$.
\end{definition}

In the following we find the general form of the co-atom covers of ${\mathscr U}(X)$.

\begin{lemma}\label{PJE}
Let $X$ be a locally pseudocompact space and let $Y\in{\mathscr U}(X)$. The following are equivalent:
\begin{itemize}
\item[\rm(1)] $Y$ is a co-atom cover in ${\mathscr U}(X)$.
\item[\rm(2)] $Y=e_X(\{a,b\})$ for some $a\in\mathrm{cl}_{\beta X}(\beta X\setminus\upsilon X)$ and $b\in(\beta X\setminus X)\setminus\mathrm{cl}_{\beta X}(\beta X\setminus\upsilon X)$.
\end{itemize}
\end{lemma}

\begin{proof}
(1) {\em implies} (2). First, we show that there exists some $F\in{\mathscr F}(Y)$ with $|F|>1$. Suppose otherwise, that is, suppose that the elements of ${\mathscr F}(X)$ are all singletons. Let
\[{\mathscr A}={\mathscr F}(Y)\setminus\big\{\{y\}:y\in\mathrm{cl}_{\beta X}(\beta X\setminus\upsilon X)\big\}.\]
Note that by Lemma \ref{ASD} we have
\[\big\{\{y\}:y\in\mathrm{cl}_{\beta X}(\beta X\setminus\upsilon X)\big\}\subsetneqq{\mathscr F}(Y),\]
where the proper inclusion is because $Y\neq\zeta X$. Thus ${\mathscr A}$ is non-empty. Either $|{\mathscr A}|=1$ or $|{\mathscr A}|>1$. In the first case, ${\mathscr A}=\{\{c\}\}$, where
\[c\in(\beta X\setminus X)\setminus\mathrm{cl}_{\beta X}(\beta X\setminus\upsilon X).\]
But (by Lemma \ref{DFH}) this implies that $Y=e_X(\{c\})$, which is not possible, as $Y$ cannot be a co-atom. In the second case, there exist some distinct
\[c,d\in(\beta X\setminus X)\setminus\mathrm{cl}_{\beta X}(\beta X\setminus\upsilon X).\]
But then
\[Y<e_X\big(\{c\}\big)<\zeta X\mbox{ and }Y<e_X\big(\{d\}\big)<\zeta X,\]
which is again not possible. Let $F\in{\mathscr F}(Y)$ be such that $|F|>1$. We show that such an $F$ is necessarily unique. Otherwise, there exists some $G\in{\mathscr F}(Y)$ with $|G|>1$ distinct from $F$. Then
\[Y<e_X(F)<\zeta X\mbox{ and }Y<e_X(G)<\zeta X,\]
which is not possible. This shows that $Y=e_X(F)$. To show (2), we need to show that
\[F\setminus\mathrm{cl}_{\beta X}(\beta X\setminus\upsilon X)\mbox{ and }F\cap\mathrm{cl}_{\beta X}(\beta X\setminus\upsilon X)\]
are both singletons. Suppose that $F\subseteq\mathrm{cl}_{\beta X}(\beta X\setminus\upsilon X)$. Then obviously $|F|>2$, as $e_X(F)$ cannot be a co-atom. Choose some distinct $a,b,c\in F$. Then
\[Y<e_X\big(\{a,b\}\big)<\zeta X\mbox{ and }Y<e_X\big(\{a,c\}\big)<\zeta X,\]
which is not possible. Thus $F\setminus\mathrm{cl}_{\beta X}(\beta X\setminus\upsilon X)$ is non-empty. Now, if there exist some distinct
\[e,f\in F\setminus\mathrm{cl}_{\beta X}(\beta X\setminus\upsilon X),\]
then
\[Y<e_X\big(\{e\}\big)<\zeta X\mbox{ and }Y<e_X\big(\{f\}\big)<\zeta X,\]
which is not possible. This shows that $F\setminus\mathrm{cl}_{\beta X}(\beta X\setminus\upsilon X)$ is a singleton. Let
\[F\setminus\mathrm{cl}_{\beta X}(\beta X\setminus\upsilon X)=\{b\}.\]
Next, we verify that $F\cap\mathrm{cl}_{\beta X}(\beta X\setminus\upsilon X)$ is a singleton (it is obviously non-empty, as $|F|>1$). But this is obvious, as if there exist some distinct
\[c,d\in F\cap\mathrm{cl}_{\beta X}(\beta X\setminus\upsilon X),\]
then
\[Y<e_X\big(\{c,d\}\big)<\zeta X\mbox{ and }Y<e_X\big(\{b\}\big)<\zeta X,\]
which is not possible. Let
\[F\setminus\mathrm{cl}_{\beta X}(\beta X\setminus\upsilon X)=\{a\}.\]
Then $F=\{a,b\}$ and therefore $Y=e_X(\{a,b\})$.

(2) {\em implies} (1). Note that if $Y<T<\zeta X$ for some $T\in{\mathscr U}(X)$, then each element of ${\mathscr F}(T)$ is a singleton. Thus $T=e_X(\{b\})$ and $Y$ is a co-atom cover in ${\mathscr U}(X)$.
\end{proof}

\begin{lemma}\label{JUHD}
Let $X$ be a locally pseudocompact space and let $Y=e_X(\{a,b\})$, where
\[a\in\mathrm{cl}_{\beta X}(\beta X\setminus\upsilon X)\mbox{ and }b\in(\beta X\setminus X)\setminus\mathrm{cl}_{\beta X}(\beta X\setminus\upsilon X).\]
Then $Y'=e_X(\{b\})$.
\end{lemma}

\begin{proof}
That $Y$ is a co-atom cover follows from Lemma \ref{PJE}. Also, the proof of part (2)$\Rightarrow$(1) in Lemma \ref{PJE} shows that $Y'=e_X(\{b\})$.
\end{proof}

\begin{definition}\label{FAS}
For a space $X$ denote by ${\mathscr U}^*(X)$ the set of all one-point pseudocompactifications of $X$.
\end{definition}

The next result, together with Lemma \ref{GFS}, gives an order-theoretic characterization of the elements of ${\mathscr U}^*(X)$.

\begin{lemma}\label{OSDR}
Let $X$ be a locally pseudocompact space and let $Y\in{\mathscr U}(X)$. The following are equivalent:
\begin{itemize}
\item[\rm(1)] $Y\in{\mathscr U}^*(X)$.
\item[\rm(2)] $Y$ satisfies the following:
\begin{itemize}
\item[\rm(a)] $Y\leq T$ for every co-atom $T$ in ${\mathscr U}(X)$ of type $(\mathrm{II})$.
\item[\rm(b)] $Y\leq D$ for every co-atom cover $D$ in ${\mathscr U}(X)$ such that $Y\leq D'$.
\end{itemize}
\end{itemize}
\end{lemma}

\begin{proof}
(1) {\em implies} (2). Note that $Y=e_X(F)$ for some $F$ containing $\mathrm{cl}_{\beta X}(\beta X\setminus\upsilon X)$; see Lemma \ref{ASD}. To show (2.a), let $T$ be a co-atom in ${\mathscr U}(X)$ of type $(\mathrm{II})$. Then $T=e_X(\{a,b\})$ for some distinct $a,b\in\mathrm{cl}_{\beta X}(\beta X\setminus\upsilon X)$. Now $\{a,b\}\subseteq F$ and therefore $T\geq Y$ (by Lemma \ref{DFH}). To show (2.b), let $D$ be a co-atom cover in ${\mathscr U}(X)$ with $D'\geq Y$. By Lemma \ref{PJE} we have $D=e_X(\{c,d\})$, where
\[c\in\mathrm{cl}_{\beta X}(\beta X\setminus\upsilon X)\mbox{ and }d\in(\beta X\setminus X)\setminus\mathrm{cl}_{\beta X}(\beta X\setminus\upsilon X).\]
Then $D'=e_X(\{d\})$, by Lemma \ref{JUHD}, and therefore $d\in F$. But also $c\in F$, and thus $\{c,d\}\subseteq F$. Therefore $D\geq Y$.

(2) {\em implies} (1). To show (1), it suffices to show that ${\mathscr F}(Y)$ contains only one element. Suppose to the contrary that there exist some distinct $F,G\in{\mathscr F}(Y)$. Consider the following cases.
\begin{description}
\item[{\sc Case 1}] Suppose that
\[F\cap\mathrm{cl}_{\beta X}(\beta X\setminus\upsilon X)\mbox{ and }G\cap\mathrm{cl}_{\beta X}(\beta X\setminus\upsilon X)\]
are both non-empty. Let
\[a\in F\cap\mathrm{cl}_{\beta X}(\beta X\setminus\upsilon X)\mbox{ and }b\in G\cap\mathrm{cl}_{\beta X}(\beta X\setminus\upsilon X).\]
Now $e_X(\{a,b\})$ is a co-atom in ${\mathscr U}(X)$ of type $(\mathrm{II})$, and thus, by our assumption $e_X(\{a,b\})\geq Y$. Therefore $\{a,b\}\subseteq H$ for some $H\in{\mathscr F}(Y)$. Since distinct elements in ${\mathscr F}(Y)$ are disjoint, this implies that $F=H=G$, which is a contradiction.
\item[{\sc Case 2}] Suppose that either
\[F\cap\mathrm{cl}_{\beta X}(\beta X\setminus\upsilon X)\mbox{ or }G\cap\mathrm{cl}_{\beta X}(\beta X\setminus\upsilon X),\]
say the latter, is empty. Let $c\in G$ and choose some $d\in\mathrm{cl}_{\beta X}(\beta X\setminus\upsilon X)$. Then $D=e_X(\{c,d\})$ is a co-atom cover in ${\mathscr U}(X)$ with $D'=e_X(\{c\})\geq Y$; see Lemmas \ref{PJE} and \ref{JUHD}. Therefore by our assumption $D\geq Y$. But this implies that $\{c,d\}\subseteq H$ for some $H\in{\mathscr F}(Y)$. Again, since distinct elements in ${\mathscr F}(Y)$ are disjoint, this implies that $H=G$, which is a contradiction, as $d\notin G$ by the choice of $d$.
\end{description}
\end{proof}

\begin{lemma}\label{EWER}
Let $X$ and $Y$ be locally pseudocompact non-pseudocompact spaces. Let
\[\Theta:{\mathscr U}(X)\rightarrow{\mathscr U}(Y)\]
be an order-isomorphism. Let $T\in{\mathscr U}(X)$. Then
\begin{itemize}
\item[\rm(1)] If $T$ is a co-atom (of type $(\mathrm{I})$, of type $(\mathrm{II})$, respectively) in ${\mathscr U}(X)$ then $\Theta(T)$ is a co-atom (of type $(\mathrm{I})$, of type $(\mathrm{II})$, respectively) in ${\mathscr U}(Y)$.
\item[\rm(2)] If $T$ is a co-atom cover in ${\mathscr U}(X)$ then $\Theta(T)$ is a co-atom cover in ${\mathscr U}(Y)$ and $\Theta(T')=(\Theta(T))'$.
\item[\rm(3)] If $T\in{\mathscr U}^*(X)$ then $\Theta(T)\in{\mathscr U}^*(Y)$.
\end{itemize}
\end{lemma}

\begin{proof}
The lemma follows from Lemmas \ref{GFS} and \ref{OSDR} and the definitions.
\end{proof}

\begin{definition}\label{OPS}
For a space $X$ denote by ${\mathscr C}(X)$ the set of all closed subspaces of $X$.
\end{definition}

\begin{lemma}\label{GJK}
Let $X$ be a locally compact non-pseudocompact space. Then there exists an anti-order-isomorphism
\[\Lambda:\big(\big\{C\in{\mathscr C}(\beta X\setminus X):\mathrm{cl}_{\beta X}(\beta X\setminus\upsilon X)\subseteq C\big\},\subseteq\big)\rightarrow\big({\mathscr U}^*(X),\leq\big).\]
\end{lemma}

\begin{proof}
Let $C$ be a closed subspace of $\beta X\setminus X$ containing $\mathrm{cl}_{\beta X}(\beta X\setminus\upsilon X)$. Note that $C$ is compact, as it is closed in $\beta X\setminus X$ and the latter is compact (as $X$ is locally compact). Let $T$ be the quotient space of $\beta X$ obtained by contracting $C$ to a point $p$. Define $\Lambda(C)=X\cup\{p\}$, considered as a subspace of $T$. Then $\Lambda(C)\in{\mathscr U}^*(X)$, by Lemma \ref{ASD}. By Lemma \ref{DFH} the mapping $\Lambda$ is an order-isomorphism onto its image, as ${\mathscr F}(\Lambda(C))=\{C\}$. That $\Lambda$ is surjective is obvious and follows again from Lemma \ref{DFH}.
\end{proof}

Observe that for a space $X$ we have
\[(\beta X\setminus X)\setminus\mathrm{cl}_{\beta X}(\beta X\setminus\upsilon X)=\mathrm{int}_{\beta X}\upsilon X\setminus X.\]

The following lemma is known (see Theorem 5.3 ((2)$\Rightarrow$(1)) of \cite{Ko4}); the proof is included here for the sake of completeness.

\begin{lemma}\label{RREQ}
Let $X$ and $Y$ be locally compact non-pseudocompact spaces. If ${\mathscr U}^*(X)$ and ${\mathscr U}^*(Y)$ are order-isomorphic then $(\beta X\setminus X)\setminus\mathrm{cl}_{\beta X}(\beta X\setminus\upsilon X)$ and $(\beta Y\setminus Y)\setminus\mathrm{cl}_{\beta Y}(\beta Y\setminus\upsilon Y)$ are homeomorphic.
\end{lemma}

\begin{proof}
By Lemma \ref{GJK} there exists an order-isomorphism $F$ from the set of all closed subspaces of $\beta X\setminus X$ containing $\mathrm{cl}_{\beta X}(\beta X\setminus\upsilon X)$ to the set of all closed subspaces of $\beta Y\setminus Y$ containing $\mathrm{cl}_{\beta Y}(\beta Y\setminus\upsilon Y)$, both partially ordered by $\subseteq$. Note that $X\subseteq\mathrm{int}_{\beta X}\upsilon X$ and $Y\subseteq\mathrm{int}_{\beta Y}\upsilon Y$ by Lemma \ref{FGA}, as $X$ and $Y$ are both locally compact and therefore locally pseudocompact. We define an order-isomorphism
\[f:{\mathscr C}(\mathrm{int}_{\beta X}\upsilon X\setminus X)\rightarrow{\mathscr C}(\mathrm{int}_{\beta Y}\upsilon Y\setminus Y);\]
this will prove the lemma. Let $A\in{\mathscr C}(\mathrm{int}_{\beta X}\upsilon X\setminus X)$. Let $A'$ be a closed subspace of $\beta X\setminus X$ with
\[A=A'\cap\mathrm{int}_{\beta X}\upsilon X.\]
Note that $A'$ is compact, as it is closed in $\beta X\setminus X$ and the latter is compact, since $X$ is locally compact. Thus
\[A\cup\mathrm{cl}_{\beta X}(\beta X\setminus\upsilon X)=A'\cup\mathrm{cl}_{\beta X}(\beta X\setminus\upsilon X)\]
is compact and non-empty (as $X$ is non-pseudocompact). Define
\[f(A)=\mathrm{int}_{\beta Y}\upsilon Y\cap F\big(A\cup\mathrm{cl}_{\beta X}(\beta X\setminus\upsilon X)\big).\]
Clearly, $f$ is well-defined. That $f$ is an order-homomorphism is straightforward. Now, let
\[g:{\mathscr C}(\mathrm{int}_{\beta Y}\upsilon Y\setminus Y)\rightarrow{\mathscr C}(\mathrm{int}_{\beta X}\upsilon X\setminus X)\]
be defined by
\[g(B)=\mathrm{int}_{\beta X}\upsilon X\cap F^{-1}\big(B\cup\mathrm{cl}_{\beta Y}(\beta Y\setminus\upsilon Y)\big)\]
for any $B\in{\mathscr C}(\mathrm{int}_{\beta Y}\upsilon Y\setminus Y)$. If $A\in{\mathscr C}(\mathrm{int}_{\beta X}\upsilon X\setminus X)$ then
\begin{eqnarray*}
f(A)\cup\mathrm{cl}_{\beta Y}(\beta Y\setminus\upsilon Y)&=&\big(\mathrm{int}_{\beta Y}\upsilon Y\cap F\big(A\cup\mathrm{cl}_{\beta X}(\beta X\setminus\upsilon X)\big)\big)\cup\mathrm{cl}_{\beta Y}(\beta Y\setminus\upsilon Y)\\&=&F\big(A\cup\mathrm{cl}_{\beta X}(\beta X\setminus\upsilon X)\big)
\end{eqnarray*}
and therefore
\begin{eqnarray*}
g\big(f(A)\big)&=&\mathrm{int}_{\beta X}\upsilon X\cap F^{-1}\big(f(A)\cup\mathrm{cl}_{\beta Y}(\beta Y\setminus\upsilon Y)\big)\\&=&\mathrm{int}_{\beta X}\upsilon X\cap F^{-1}\big(F\big(A\cup\mathrm{cl}_{\beta X}(\beta X\setminus\upsilon X)\big)\big)\\&=&\mathrm{int}_{\beta X}\upsilon X\cap\big(A\cup\mathrm{cl}_{\beta X}(\beta X\setminus\upsilon X)\big)=A.
\end{eqnarray*}
Thus $gf$ is the identity mapping, and similarly, so is $fg$. Therefore $g=f^{-1}$. Since $g$  is an order-homomorphism, $f$ is an  order-isomorphism.
\end{proof}

We now prove our second main result.

\begin{theorem}\label{RTRT}
Let $X$ and $Y$ be locally compact non-pseudocompact spaces. If ${\mathscr U}(X)$ and ${\mathscr U}(Y)$ are order-isomorphic then $(\beta X\setminus X)\setminus\mathrm{cl}_{\beta X}(\beta X\setminus\upsilon X)$ and $(\beta Y\setminus Y)\setminus\mathrm{cl}_{\beta Y}(\beta Y\setminus\upsilon Y)$ are homeomorphic.
\end{theorem}

\begin{proof}
This follows from Lemmas \ref{EWER} and \ref{RREQ}, as every order-isomorphism of ${\mathscr U}(X)$ onto ${\mathscr U}(Y)$ restricts to an order-isomorphism of ${\mathscr U}^*(X)$ onto ${\mathscr U}^*(Y)$.
\end{proof}

\begin{question}\label{HHLTD}
In Theorem \ref{RTRT}, is it possible to replace local compactness by local pseudocompactness? Also, if we denote by ${\mathscr A}(X)$ the set of all pseudocompactifications of a space $X$, is it then possible to replace ${\mathscr U}(X)$ and ${\mathscr U}(Y)$ by ${\mathscr A}(X)$ and ${\mathscr A}(Y)$, respectively?
\end{question}

\section{Realcompactifications with compact remainder}

Realcompact extensions are called {\em realcompactifications}. (Recall that a space is said to be {\em realcompact} if it is homeomorphic to a closed subspace of a product of copies of the real line.) In this section we consider realcompactifications with compact remainder, with results dual to those of the previous section. The section is divided into two parts. In the first part we describe the general form of all realcompactifications of a given space $X$ with compact remainder. In the second part we deal with the partially ordered set of all realcompactifications of a space $X$ with compact remainder. As in the case of pseudocompactifications, we show that this partially ordered set determines the topology of certain subspaces of $\beta X\setminus X$.

\subsection{Realcompactifications with compact remainder; their general form.}
The results of this part are analogous to those in Part \ref{IUI}.

\begin{definition}\label{TDFK}
For a space $X$ denote by ${\mathscr R}(X)$ the set of all realcompactifications of $X$ with compact remainder.
\end{definition}

The following is the counterpart of Lemma \ref{ASD}.

\begin{lemma}\label{SEDD}
Let $X$ be a space, let $Y$ be an extension of $X$ with compact remainder, let $\zeta Y$ be a compactification of $Y$ and let $\phi:\beta X\rightarrow\zeta Y$ continuously extend $\mathrm{id}_X$. The following are equivalent:
\begin{itemize}
\item[\rm(1)] $Y\in{\mathscr R}(X)$.
\item[\rm(2)] $\mathrm{cl}_{\beta X}(\upsilon X\setminus X)\subseteq\phi^{-1}[Y\setminus X]$.
\end{itemize}
\end{lemma}

\begin{proof}
As in the proof of Lemma \ref{ASD}, it suffices that we consider the case when $\zeta Y=\beta Y$ and replace condition (2) by the requirement that $\upsilon X\setminus X\subseteq\phi^{-1}[Y\setminus X]$.

(1) {\em implies} (2). Let $x\in\upsilon X\setminus X$ and suppose to the contrary that $x\notin\phi^{-1}[Y\setminus X]$. Then $x\in \beta Y\setminus Y$. (Recall the construction of $\beta Y$ given in Lemma \ref{j2}.) Since $Y$ is realcompact, there exists some $Z\in{\mathscr Z}(\beta Y)$ such that $x\in Z$ and $Z\cap Y$ is empty. Now $\phi^{-1}[Z]\in{\mathscr Z}(\beta X)$ misses $X$ and contains $x$, contradicting the fact that $x\in\upsilon X$.

(2) {\em implies} (1). Suppose to the contrary that $Y$ is not realcompact. There exists some $p\in\upsilon Y\setminus Y$. Then $p\in\beta X\setminus X$ and $p\notin\phi^{-1}[Y\setminus X]$. (See Lemma \ref{j2}.) Thus $p\notin\upsilon X\setminus X$ and therefore $p\notin\upsilon X$. Let $S\in{\mathscr Z}(\beta X)$ be such that $p\in S$ and $S\cap X$ is empty. Let $T\in{\mathscr Z}(\beta X)$ be such that $p\in T$ and $T\cap\phi^{-1}[Y\setminus X]$ is empty. Now
\[G=(S\cap T)\setminus\phi^{-1}[Y\setminus X]\]
contains $p$ and it is a countable intersection of open subspaces of $\beta X$ each missing $\phi^{-1}[Y\setminus X]$. Thus (using Lemma \ref{j2}) $G$ is a non-empty $G_\delta$-set of $\beta Y$ which misses $Y$, contradicting the fact that $p\in\upsilon Y$.
\end{proof}

\begin{definition}\label{AQL}
A space $X$ is called {\em locally realcompact} if every $x\in X$ has an open neighborhood $U$ in $X$ with realcompact closure $\mathrm{cl}_X U$.
\end{definition}

Recall that realcompactness is hereditary with respect to closed subspaces; thus, a space with a realcompactification with compact remainder is locally realcompact.

\begin{lemma}[Mack, Rayburn and Woods \cite{MRW1}]\label{KJD}
A space $X$ is locally realcompact if and only if $X$ is open in $\upsilon X$.
\end{lemma}

The following is the counterpart of Lemma \ref{FGA}.

\begin{lemma}\label{ASL}
A space $X$ is locally realcompact if and only if
\[X\subseteq\beta X\setminus\mathrm{cl}_{\beta X}(\upsilon X\setminus X).\]
\end{lemma}

\begin{proof}
By Lemma \ref{KJD}, the space $X$ is locally realcompact if and only if $X$ is open in $\upsilon X$. Thus, $X$ is locally realcompact if and only if $\upsilon X\setminus X$ is closed in $\upsilon X$, if and only if $\mathrm{cl}_{\upsilon X}(\upsilon X\setminus X)\subseteq\upsilon X\setminus X$, if and only if $X\cap\mathrm{cl}_{\upsilon X}(\upsilon X\setminus X)$ is empty, if and only if $X\cap\mathrm{cl}_{\beta X}(\upsilon X\setminus X)$ is empty, if and only if $X\subseteq\beta X\setminus\mathrm{cl}_{\beta X}(\upsilon X\setminus X)$.
\end{proof}

The following is the counterpart of Theorem \ref{RES}.

\begin{theorem}\label{DRES}
Let $X$ be a locally realcompact space. Let $\zeta X$ be a compactification of $X$, let $\phi:\beta X\rightarrow\zeta X$ continuously extend $\mathrm{id}_X$ and let $E$ be a compact subspace of $\zeta X\setminus X$ containing $\phi[\mathrm{cl}_{\beta X}(\upsilon X\setminus X)]$. Then $Y=X\cup E\in{\mathscr R}(X)$ (considered as a subspace of $\zeta X$). Furthermore, every element of ${\mathscr R}(X)$ is of this form.
\end{theorem}

\subsection{Realcompactifications with compact remainder; their partially ordered sets.}
The results of this part are analogous to those in Part \ref{TIS}. Theorems \ref{AEF} and \ref{PKS} are dual to Theorems  \ref{RSQ} and \ref{RTRT}, respectively, with analogous proofs. One should simply replace $\mathrm{cl}_{\beta X}(\beta X\setminus\upsilon X)$ by $\mathrm{cl}_{\beta X}(\upsilon X\setminus X)$ in all proofs and note the duality between Lemmas \ref{ASD} and \ref{SEDD} and Lemmas \ref{FGA} and \ref{ASL}.

\begin{definition}\label{TER}
For a space $X$ let
\[\rho X=X\cup\mathrm{cl}_{\beta X}(\upsilon X\setminus X)\]
considered as a subspace of $\beta X$.
\end{definition}

\begin{proposition}\label{OPJ}
Let $X$ be a locally realcompact space. Then ${\mathscr R}(X)$ has a largest element, namely $\rho X$.
\end{proposition}

\begin{theorem}\label{AEF}
Let $X$ and $Y$ be locally realcompact non-realcompact spaces. If ${\mathscr R}(X)$ and ${\mathscr R}(Y)$ are order-isomorphic then $\mathrm{cl}_{\beta X}(\upsilon X\setminus X)$ and $\mathrm{cl}_{\beta Y}(\upsilon Y\setminus Y)$ are homeomorphic.
\end{theorem}

In Theorem \ref {AEF}, similar to its dual result Theorems \ref{RSQ}, the converse does not holds. This is shown by the following example.

\begin{example}\label{QFKJ}
Let $X=[0,\Omega)$ and let $Y=D(\aleph_0)\oplus[0,\Omega)$. (Where $D(\aleph_0)$ is the discrete space of cardinality $\aleph_0$, $\Omega$ is the first uncountable ordinal and $\oplus$ denotes the disjoint union.) Then
\[\beta Y=\beta\big(D(\aleph_0)\big)\oplus[0,\Omega]\mbox{ and }\upsilon Y=D(\aleph_0)\oplus[0,\Omega].\]
Thus $\mathrm{cl}_{\beta X}(\upsilon X\setminus X)=\{\Omega\}$ and $\mathrm{cl}_{\beta Y}(\upsilon Y\setminus Y)=\{\Omega\}$ are homeomorphic, while ${\mathscr R}(X)$ and ${\mathscr R}(Y)$ are not order-isomorphic, as ${\mathscr R}(X)$ consists of only a single element, whereas ${\mathscr R}(Y)$ is infinite.
\end{example}

\begin{question}\label{TD}
For a space $X$ let ${\mathscr A}(X)$ denote the set of all realcompactifications of $X$. For a (locally realcompact) space $X$, does the order structure of ${\mathscr A}(X)$ (partially ordered by $\leq$) determine the topology of $\mathrm{cl}_{\beta X}(\upsilon X\setminus X)$?
\end{question}

\begin{theorem}\label{PKS}
Let $X$ and $Y$ be locally compact non-realcompact spaces. If ${\mathscr R}(X)$ and ${\mathscr R}(Y)$ are order-isomorphic then $(\beta X\setminus X)\setminus\mathrm{cl}_{\beta X}(\upsilon X\setminus X)$ and $(\beta Y\setminus Y)\setminus\mathrm{cl}_{\beta Y}(\upsilon Y\setminus Y)$ are homeomorphic.
\end{theorem}

\begin{question}\label{OTD}
In Theorem \ref{PKS}, is it possible to replace local compactness by local realcompactness? Also, if we denote by ${\mathscr A}(X)$ the set of all realcompactifications of a space $X$, is it then possible to replace ${\mathscr R}(X)$ and ${\mathscr R}(Y)$ by ${\mathscr A}(X)$ and ${\mathscr A}(Y)$, respectively?
\end{question}

\section{$\mathfrak{P}$-extensions with compact remainder}

Let $X$ be a space and let $\mathfrak{P}$ be a topological property. An extension $Y$ of $X$ is called a {\em $\mathfrak{P}$-extension} of $X$ if it has $\mathfrak{P}$. In this section we consider the set of all $\mathfrak{P}$-extensions of a space $X$ with compact remainder (where $\mathfrak{P}$ is subject to ceratin mild requirements) and study its order structure by relating it to the topologies of certain subspaces of the outgrowth $\beta X\setminus X$.

\begin{definition}\label{DSAS}
Let $X$ be a space and let $\mathfrak{P}$ be a topological property. We denote by ${\mathscr E}_\mathfrak{P}(X)$ the set of all $\mathfrak{P}$-extensions of $X$ with compact remainder.
\end{definition}

Let $\mathfrak{P}$ be a topological property. Then
\begin{itemize}
  \item $\mathfrak{P}$ is {\em closed} ({\em open}, respectively) {\em hereditary}, if any closed (open, respectively) subspace of a space with $\mathfrak{P}$, has $\mathfrak{P}$.
  \item $\mathfrak{P}$ is {\em finitely additive}, if any space which is expressible as a finite disjoint union of its closed subspaces each with $\mathfrak{P}$, has $\mathfrak{P}$.
  \item $\mathfrak{P}$ is {\em invariant under perfect mappings} ({\em  inverse invariant under perfect mappings}, respectively) if for every perfect surjective mapping $f:X\rightarrow Y$, the space $Y$ ($X$, respectively) has $\mathfrak{P}$, provided that $X$ ($Y$, respectively) has $\mathfrak{P}$. If $\mathfrak{P}$ is both invariant and inverse invariant under perfect mappings then it is called {\em perfect}. (Recall that a closed continuous mapping $f:X\rightarrow Y$ is {\em perfect}, if each fiber $f^{-1}(y)$, where $y\in Y$, is a compact subspace of $X$.)
  \item $\mathfrak{P}$ {\em satisfies Mr\'{o}wka's condition $(\mathrm{W})$} if it satisfies the following: If $X$ is a space in which there
      exists a point $p$ with an open base ${\mathscr B}$ for $X$ at $p$ such that $X\setminus  B$ has $\mathfrak{P}$ for each $B\in{\mathscr B}$, then $X$ has $\mathfrak{P}$. (See \cite{Mr}.)
\end{itemize}

\begin{remark}\label{AEKF}
If $\mathfrak{P}$ is a topological property which is closed hereditary and productive, then Mr\'{o}wka's condition $(\mathrm{W})$ is equivalent to the following condition: If a space $X$ is the union of a compact space and a space with $\mathfrak{P}$, then $X$ has $\mathfrak{P}$. (See \cite{MRW1}.)
\end{remark}

Recall that a subspace of a space is said to be {\em clopen} if it is simultaneously closed and open.

\begin{definition}\label{PDE}
We call a topological property a {\em compactness-like topological property} if it is clopen hereditary, finitely additive, perfect and satisfies Mr\'{o}wka's condition $(\mathrm{W})$.
\end{definition}

\begin{example}\label{QLL}
The list of compactness-like topological properties is quite long and includes almost all important covering properties (that is, topological properties described in terms of the existence of certain kinds of open subcovers or refinements of a given open cover of a certain type), among them are: compactness, countable compactness (more generally, $[\theta,\kappa]$-compactness), the Lindel\"{o}f property (more generally, the $\mu$-Lindel\"{o}f property), paracompactness, metacompactness, countable paracompactness, subparacompactness, submetacompactness (or $\theta$-refinability), the $\sigma$-para-Lindel\"{o}f property and also $\alpha$-boundedness. (See \cite{Ko3} for the proofs and \cite{Bu}, \cite{Steph} and \cite{Va} for the definitions.)
\end{example}

Let $\mathfrak{P}$ be a topological property. Then $\mathfrak{P}$ is said to be {\em preserved under finite closed sums}, if any space which is expressible as a finite union of its closed subspaces each having $\mathfrak{P}$, also has $\mathfrak{P}$. It is known that any finitely additive topological property which is invariant under perfect mappings is preserved under finite closed sums. (See Theorem 3.7.22 of \cite{E}.) Also, it is known that any topological property which is hereditary with respect to clopen subspaces and is inverse invariant under perfect mappings, is hereditary with respect to closed subspaces. (See Theorem 3.7.29 of \cite{E}.) Thus, in particular, any compactness-like topological property is closed hereditary and is preserved under finite closed sums. We may use this fact without explicitly referring to it.

The following subspace of $\beta X$, introduced and studied in \cite{Ko3} (see also \cite{Ko4}, \cite{Ko5} and \cite{Ko9}), plays a crucial role in what follows.

\begin{definition}\label{XA}
For a space $X$ and a topological property  $\mathfrak{P}$, let
\[\lambda_\mathfrak{P} X=\bigcup\big\{\mathrm{int}_{\beta X} \mathrm{cl}_{\beta X}C:C\in\mathrm{Coz}(X)\mbox{ and } \mathrm{cl}_X C \mbox{ has }\mathfrak{P}\big\}.\]
\end{definition}

If $X$ is a space and $D$ is a dense subspace of $X$, then
\[\mathrm{cl}_XU=\mathrm{cl}_X(U\cap D)\]
for every open subspace $U$ of $X$. We use the following simple observation in a number of places.

\begin{lemma}\label{LKG}
Let $X$ be a space. If $f:X\rightarrow[0,1]$ is continuous and $0<r<1$ then
\[f_\beta^{-1}\big[[0,r)\big]\subseteq\mathrm{int}_{\beta X}\mathrm{cl}_{\beta X}f^{-1}\big[[0,r)\big].\]
\end{lemma}

\begin{proof}
Note that
\[f_\beta^{-1}\big[[0,r)\big]\subseteq\mathrm{int}_{\beta X}\mathrm{cl}_{\beta X}f_\beta^{-1}\big[[0,r)\big].\]
On the other hand, since $X$ is dense in $\beta X$, we have
\[\mathrm{cl}_{\beta X}f_\beta^{-1}\big[[0,r)\big]=\mathrm{cl}_{\beta X}\big(X\cap f_\beta^{-1}\big[[0,r)\big]\big)=\mathrm{cl}_{\beta X}f^{-1}\big[[0,r)\big].\]
\end{proof}

The following lemma is the counterpart of Lemmas \ref{ASD} and \ref{SEDD}. It also gives an alternative simple proof for (a special case of) Lemma 2.8 in \cite{Ko3}.

Recall that a space $X$ is called {\em locally-$\mathfrak{P}$}, when $\mathfrak{P}$ is a topological property, if each $x\in X$ has an open neighborhood $U$ in $X$ whose closure $\mathrm{cl}_XU$ has $\mathfrak{P}$.

\begin{lemma}\label{Y16}
Let $X$ be a space and let $\mathfrak{P}$ be a compactness-like topological property. Let $Y$ be an extension of $X$ with compact remainder, let $\zeta Y$ be a compactification of $Y$ and let $\phi:\beta X\rightarrow\zeta Y$ continuously extend $\mathrm{id}_X$. The following are equivalent:
\begin{itemize}
\item[\rm(1)] $Y\in{\mathscr E}_\mathfrak{P}(X)$.
\item[\rm(2)] $X$ is locally-$\mathfrak{P}$ and $\beta X\setminus\lambda_\mathfrak{P}X\subseteq\phi^{-1}[Y\setminus X]$.
\end{itemize}
\end{lemma}

\begin{proof}
We only need to prove the lemma in the case when $\zeta Y=\beta Y$. (See the proof of Lemma \ref{ASD}.)

(1) {\em implies} (2). Since $\mathfrak{P}$ is closed hereditary, the space $X$, having a $\mathfrak{P}$-extension with compact remainder, is locally-$\mathfrak{P}$. Let $x\in\beta X\setminus\lambda_\mathfrak{P}X$ and suppose to the contrary that $x\notin\phi^{-1}[Y\setminus X]$. Let $f:\beta X\rightarrow[0,1]$ be continuous with $f(x)=0$ and $f(t)=1$ for any $t\in\phi^{-1}[Y\setminus X]$. Then (using Lemma \ref{j2})
\[Z=X\cap f^{-1}\big[[0,1/2]\big]=Y\cap\phi\big[f^{-1}\big[[0,1/2]\big]\big],\]
being closed in $Y$, has $\mathfrak{P}$. Let
\[C=X\cap f^{-1}\big[[0,1/2)\big].\]
Then $C\in\mathrm{Coz}(X)$ and $\mathrm{cl}_X C$ has $\mathfrak{P}$, as it is closed in $Z$. Thus
\[\mathrm{int}_{\beta X} \mathrm{cl}_{\beta X}C\subseteq\lambda_\mathfrak{P} X\]
by the definition of $\lambda_\mathfrak{P} X$. But $x\in\mathrm{int}_{\beta X} \mathrm{cl}_{\beta X}C$, as $x\in f^{-1}[[0,1/2)]$ and \[f^{-1}\big[[0,1/2)\big]\subseteq\mathrm{int}_{\beta X}\mathrm{cl}_{\beta X}C\]
by Lemma \ref{LKG}. Therefore $x\in\lambda_\mathfrak{P} X$, which is a contradiction.

(2) {\em implies} (1). Let $T$ be the quotient space of $\beta Y$ obtained by contracting $Y\setminus X$ to a point $p$ and denote by $q:\beta Y\rightarrow T$ its quotient mapping. Note that $T$ is completely regular as $Y\setminus X$ is compact. We show that $Y^*=X\cup\{p\}$ has $\mathfrak{P}$, from this and the fact that $q|Y:Y\rightarrow Y^*$ is perfect (and surjective) it will then follow that $Y$ has $\mathfrak{P}$. To show that $Y^*$ has $\mathfrak{P}$ we verify that $Y^*\setminus W$ has $\mathfrak{P}$ for every open neighborhood $W$ of $p$ in $Y^*$. Let $W$ be an open neighborhood of $p$ in $Y^*$ and let $W'$ be open in $T$ with $W'\cap Y^*=W$. Then
\[\beta X\setminus\lambda_\mathfrak{P}X\subseteq\phi^{-1}[Y\setminus X]\subseteq\phi^{-1}\big[q^{-1}[W']\big]\]
and thus
\[\beta X\setminus\phi^{-1}\big[q^{-1}[W']\big]\subseteq\lambda_\mathfrak{P}X.\]
By compactness and the definition of $\lambda_\mathfrak{P}X$ we have
\[\beta X\setminus\phi^{-1}\big[q^{-1}[W']\big]\subseteq\mathrm{int}_{\beta X} \mathrm{cl}_{\beta X}C_1\cup\cdots\cup\mathrm{int}_{\beta X} \mathrm{cl}_{\beta X}C_n\]
where $C_i\in\mathrm{Coz}(X)$ and $\mathrm{cl}_X C_i$ has $\mathfrak{P}$ for each $i=1,\ldots,n$. Now
\[Y^*\setminus W=\big(\beta X\setminus\phi^{-1}\big[q^{-1}[W']\big]\big)\cap X\subseteq\mathrm{cl}_XC_1\cup\cdots\cup\mathrm{cl}_XC_n\]
and the latter, being a finite union of its closed subspaces each with $\mathfrak{P}$, has $\mathfrak{P}$, thus its closed subspace $Y^*\setminus W$ also has $\mathfrak{P}$.
\end{proof}

The following lemma, which is the counterpart of Lemmas \ref{FGA} and \ref{KJD}, is a slight modification of Lemma 2.10 of \cite{Ko3}.

\begin{lemma}\label{15}
Let $X$ be a space and let $\mathfrak{P}$ be a compactness-like topological property. Then $X\subseteq\lambda_\mathfrak{P}X$ if and only if $X$ is locally-$\mathfrak{P}$.
\end{lemma}

\begin{proof}
Suppose that $X$ is locally-$\mathfrak{P}$. Let $x\in X$ and let $U$ be an open neighborhood of $x$ in $X$ whose closure $\mathrm{cl}_XU$ has $\mathfrak{P}$. Let $f:X\rightarrow[0,1]$ be continuous with $f(x)=0$ and $f(t)=1$ for any $t\in X\setminus U$. Let
\[C=f^{-1}\big[[0,1/2)\big]\in\mathrm{Coz}(X).\]
Then $C\subseteq U$ and thus $\mathrm{cl}_XC$ has $\mathfrak{P}$, as it is closed in $\mathrm{cl}_XU$. Therefore
\[\mathrm{int}_{\beta X}\mathrm{cl}_{\beta X}C\subseteq\lambda_\mathfrak{P}X.\]
But then $x\in\lambda_\mathfrak{P} X$, as $x\in f_\beta^{-1}[[0,1/2)]$ and
\[f_\beta^{-1}\big[[0,1/2)\big]\subseteq\mathrm{int}_{\beta X}\mathrm{cl}_{\beta X}C\]
by Lemma \ref{LKG}.

For the converse, suppose that $X\subseteq\lambda_\mathfrak{P} X$. Let $x\in X$. Then $x\in\lambda_\mathfrak{P} X$ and therefore $x\in\mathrm{int}_{\beta X}\mathrm{cl}_{\beta X}C$ for some $C\in\mathrm{Coz}(X)$ such that $\mathrm{cl}_XC$ has $\mathfrak{P}$. Let
\[V=X\cap\mathrm{int}_{\beta X}\mathrm{cl}_{\beta X}C.\]
Then $V$ is an open neighborhood of $x$ in $X$ and since $V\subseteq \mathrm{cl}_XC$, the set $\mathrm{cl}_XV$ has $\mathfrak{P}$, as it is closed in $\mathrm{cl}_XC$.
\end{proof}

Theorems \ref{UYF}, \ref{GFO} and \ref{PPKS} are dual to Theorems \ref{RES}, \ref{RSQ} and \ref{RTRT} (and to Theorems \ref{DRES}, \ref{AEF} and \ref{PKS}), respectively, with analogous proofs. One should simply replace $\mathrm{int}_{\beta X}\upsilon X$ by $\lambda_\mathfrak{P} X$ in all proofs and note the duality between Lemmas \ref{ASD} and \ref{Y16} and Lemmas \ref{FGA} and \ref{15}.

\begin{theorem}\label{UYF}
Let $X$ be a locally-$\mathfrak{P}$ space, where $\mathfrak{P}$ is a compactness-like topological property. Let $\zeta X$ be a compactification of $X$, let $\phi:\beta X\rightarrow\zeta X$ continuously extend $\mathrm{id}_X$ and let $E$ be a compact subspace of $\zeta X\setminus X$ containing $\phi[\beta X\setminus\lambda_\mathfrak{P}X]$. Then $Y=X\cup E\in{\mathscr E}_\mathfrak{P}(X)$ (considered as a subspace of $\zeta X$). Furthermore, every element of ${\mathscr E}_\mathfrak{P}(X)$ is of this form.
\end{theorem}

\begin{definition}\label{GSDR}
For a space $X$ and a topological property $\mathfrak{P}$, let
\[\mu_\mathfrak{P} X=X\cup(\beta X\setminus\lambda_\mathfrak{P}X)\]
considered as a subspace of $\beta X$.
\end{definition}

\begin{proposition}\label{TG}
Let $X$ be a locally-$\mathfrak{P}$ space, where $\mathfrak{P}$ is a compactness-like topological property. Then ${\mathscr E}_\mathfrak{P}(X)$ has a largest element, namely $\mu_\mathfrak{P}X$.
\end{proposition}

\begin{theorem}\label{GFO}
Let $X$ and $Y$ be locally-$\mathfrak{P}$ non-$\mathfrak{P}$ spaces, where $\mathfrak{P}$ is a compactness-like topological property. If ${\mathscr E}_\mathfrak{P}(X)$ and ${\mathscr E}_\mathfrak{P}(Y)$ are order-isomorphic then $\beta X\setminus\lambda_\mathfrak{P}X$ and $\beta Y\setminus\lambda_\mathfrak{P}Y$ are homeomorphic.
\end{theorem}

\begin{theorem}\label{PPKS}
Let $X$ and $Y$ be locally compact locally-$\mathfrak{P}$ non-$\mathfrak{P}$ spaces, where $\mathfrak{P}$ is a compactness-like topological property. If ${\mathscr E}_\mathfrak{P}(X)$ and ${\mathscr E}_\mathfrak{P}(Y)$ are order-isomorphic then $\lambda_\mathfrak{P}X\setminus X$ and $\lambda_\mathfrak{P}Y\setminus Y$ are homeomorphic.
\end{theorem}

We have seen through examples (Examples \ref{HFP} and \ref{QFKJ}) that the converses of Theorems \ref{RSQ} and \ref {AEF} do not hold in general. Analogously, we show that the converse of Theorem \ref{GFO} (the dual result of Theorems \ref{RSQ} and \ref {AEF}) does not hold either. This will be done through an example (Example \ref{DFJ}) for a specific choice of a compactness-like topological property $\mathfrak{P}$. The example (which shares several ideas of the proof of Theorem 4.36 of \cite{Ko3} and certain results from \cite{Ko4}; e.g. Lemmas 2.10, 4.1 and 4.3 of \cite{Ko4}) is very technical, and requires us to state and prove a series of lemmas preceding it. This we will do next.

To avoid ambiguity we restate Definition \ref{HGFD} in the new context.

\begin{definition}\label{HEID}
Let $X$ be a locally-$\mathfrak{P}$ space, where $\mathfrak{P}$ is a compactness-like topological property, and let $C_1,\dots,C_n$ be $n$ pairwise disjoint compact non-empty subspaces of $\beta X\setminus X$. Let $Z$ be the quotient space of $\beta X$ obtained by contracting $C_1,\dots,C_n$ to $p_1,\dots,p_n$, respectively. Define
\[e_X(C_1,\dots,C_n)=X\cup\{p_1,\dots,p_n\}\cup\big((\beta X\setminus\lambda_\mathfrak{P}X)\setminus (C_1\cup\dots\cup C_n)\big),\]
considered as a subspace of $Z$. Note that
\[e_X(C_1,\dots,C_n)\in{\mathscr E}_\mathfrak{P}(X)\]
by Lemma \ref{Y16}, and
\[{\mathscr F}\big(e_X(C_1,\dots,C_n)\big)=\{C_1,\dots,C_n\}\cup\big\{\{y\}:y\in(\beta X\setminus\lambda_\mathfrak{P}X)\setminus (C_1\cup\dots\cup C_n)\big\}.\]
\end{definition}

\begin{definition}\label{EGGH}
Let $X$ be a locally-$\mathfrak{P}$ space, where $\mathfrak{P}$ is a compactness-like topological property. A co-atom $Y$ in ${\mathscr E}_\mathfrak{P}(X)$ is
\begin{itemize}
  \item \emph{of type $(\mathrm{I})$} if $Y=e_X(\{a\})$, for some $a\in(\beta X\setminus X)\setminus(\beta X\setminus\lambda_\mathfrak{P}X)$.
  \item \emph{of type $(\mathrm{II})$} if $Y=e_X(\{a,b\})$, for some distinct $a,b\in\beta X\setminus\lambda_\mathfrak{P}X$.
\end{itemize}
\end{definition}

By an argument similar to the one we have given for Lemma \ref{UGGH} it follows that every co-atom in ${\mathscr E}_\mathfrak{P}(X)$ (where $X$ is a locally-$\mathfrak{P}$ space and $\mathfrak{P}$ is a compactness-like topological property) actually is either of type $(\mathrm{I})$ or of type $(\mathrm{II})$.

The next lemma characterizes order-theoretically the co-atoms in ${\mathscr E}_\mathfrak{P}(X)$ of type $(\mathrm{II})$ (and consequently, the co-atoms in ${\mathscr E}_\mathfrak{P}(X)$ of type $(\mathrm{I})$). The proof is analogous to its dual result Lemma \ref{GFS}. (One needs to state and prove a lemma dual to Lemma \ref{GHDS} first.) Note that in the proof of Lemma \ref{GFS} one needs only that
\[\big|\mathrm{cl}_{\beta X}(\beta X\setminus\upsilon X)\big|\geq 3.\]
This justifies the inclusion of the extra assumption in the following lemma.

\begin{lemma}\label{GGO}
Let $X$ be a locally-$\mathfrak{P}$ space, where $\mathfrak{P}$ is a compactness-like topological property, and let $T$ be a co-atom in ${\mathscr E}_\mathfrak{P}(X)$. Suppose that $|\beta X\setminus\lambda_\mathfrak{P}X|\geq 3$. The following are equivalent:
\begin{itemize}
\item[\rm(1)] $T$ is of type $(\mathrm{II})$.
\item[\rm(2)] There exists a co-atom $S$ in ${\mathscr E}_\mathfrak{P}(X)$ with
\[\big|\big\{U\in{\mathscr E}_\mathfrak{P}(X):U\geq S\wedge T\big\}\big|=5.\]
\end{itemize}
\end{lemma}

The following is dual to Definition \ref{FAS}.

\begin{definition}\label{RRS}
Let $X$ be a space and let $\mathfrak{P}$ be a compactness-like topological property. Define
\[{\mathscr E}^*_\mathfrak{P}(X)=\big\{Y\in{\mathscr E}_\mathfrak{P}(X): Y\setminus X\mbox{ is a singleton}\big\}.\]
\end{definition}

The following is dual to Lemma \ref{OSDR} with an analogous proof; it characterizes order-theoretically the elements of ${\mathscr E}_\mathfrak{P}^*(X)$ in ${\mathscr E}_\mathfrak{P}(X)$.

\begin{lemma}\label{DDR}
Let $X$ be a locally-$\mathfrak{P}$ space, where $\mathfrak{P}$ is a compactness-like topological property, and let $Y\in{\mathscr E}_\mathfrak{P}(X)$. The following are equivalent:
\begin{itemize}
\item[\rm(1)] $Y\in{\mathscr E}_\mathfrak{P}^*(X)$.
\item[\rm(2)] $Y$ satisfies the following:
\begin{itemize}
\item[\rm(a)] $Y\leq T$ for every co-atom $T$ in ${\mathscr E}_\mathfrak{P}(X)$ of type $(\mathrm{II})$.
\item[\rm(b)] $Y\leq D$ for every co-atom cover $D$ in ${\mathscr E}_\mathfrak{P}(X)$ such that $Y\leq D'$.
\end{itemize}
\end{itemize}
\end{lemma}

Recall that a space $X$ is locally compact if and only if $X$ is open in every compactification $\zeta X$ of $X$ if and only if $X$ is open in some compactification $\gamma X$ of $X$. This simple observation will be used in the proof of the following lemma which characterizes (not yet order-theoretically) the locally compact elements of ${\mathscr E}_\mathfrak{P}^*(X)$ in ${\mathscr E}_\mathfrak{P}^*(X)$.

\begin{lemma}\label{DDDS}
Let $X$ be a locally compact locally-$\mathfrak{P}$ space, where $\mathfrak{P}$ is a compactness-like topological property. Let $Y\in{\mathscr E}^*_\mathfrak{P}(X)$ and ${\mathscr F}(Y)=\{F\}$. The following are equivalent:
\begin{itemize}
\item[\rm(1)] $Y$ is locally compact.
\item[\rm(2)] $F$ is open in $\beta X\setminus X$.
\end{itemize}
\end{lemma}

\begin{proof}
Since ${\mathscr F}(e_X(F))=\{F\}$ we have $Y=e_X(F)$ by Lemma \ref{DFH}. That is, if $Z$ is the quotient space of $\beta X$ obtained by contracting $F$ to $p$, then $Y$ coincides with the subspace $X\cup\{p\}$ of $Z$. Note that $Z$ is a compactification of $Y$. Thus, $Y$ is locally compact if and only if $Y$ is open in $Z$. Let $q:\beta X\rightarrow Z$ denote the natural quotient mapping.

(1) {\em implies} (2). If $Y$ is open in $Z$ then $q^{-1}[Y]$ is open in $\beta X$. Therefore
\[F=(\beta X\setminus X)\cap q^{-1}[Y]\]
is open in $\beta X\setminus X$.

(2) {\em implies} (1). Let $W$ be an open subspace of $\beta X$ such that $W\cap(\beta X\setminus X)=F$. Since $X$ is locally compact, $X$ is open in $\beta X$. Since
\[q^{-1}\big[q[X\cup W]\big]=X\cup W\]
is open in $\beta X$ it follows that $Y=q[X\cup W]$ is open in $Z$.
\end{proof}

Our next purpose is to characterize order-theoretically the locally compact elements of ${\mathscr E}_\mathfrak{P}^*(X)$ in ${\mathscr E}_\mathfrak{P}^*(X)$. This will be done through the introduction and use of the auxiliary notion of an {\em ideal} element in ${\mathscr E}^*_\mathfrak{P}(X)$ and its order-theoretic characterization in ${\mathscr E}_\mathfrak{P}^*(X)$.

\begin{definition}\label{ERS}
Let $X$ be a space and let $\mathfrak{P}$ be a compactness-like topological property. Let $Y\in{\mathscr E}^*_\mathfrak{P}(X)$ and ${\mathscr F}(Y)=\{F\}$. Then $Y$ is called an {\em ideal} if $F\cap\lambda_\mathfrak{P}X$ is compact.
\end{definition}

\begin{lemma}\label{FSGK}
Let $X$ be a space and let $\mathfrak{P}$ be a compactness-like topological property. Then $\lambda_\mathfrak{P}X=\beta X$ if and only if $X$ has $\mathfrak{P}$.
\end{lemma}

\begin{proof}
If $X$ has $\mathfrak{P}$ then $\lambda_\mathfrak{P}X=\beta X$, as obviously $X\in\mathrm{Coz}(X)$. To show the converse, suppose that $\lambda_\mathfrak{P}X=\beta X$. By compactness and the definition of $\lambda_\mathfrak{P}X$ we have
\begin{equation}\label{FRTT}
\beta X=\mathrm{int}_{\beta X}\mathrm{cl}_{\beta X}C_1\cup\cdots\cup\mathrm{int}_{\beta X}\mathrm{cl}_{\beta X}C_n
\end{equation}
where $C_1,\ldots,C_n\in\mathrm{Coz}(X)$ and each $\mathrm{cl}_XC_1,\ldots,\mathrm{cl}_XC_n$ has $\mathfrak{P}$. Taking the intersection of both sides of (\ref{FRTT}) with $X$, we have
\[X=\mathrm{cl}_XC_1\cup\cdots\cup\mathrm{cl}_XC_n.\]
This implies that $X$ has $\mathfrak{P}$, as $\mathfrak{P}$ is preserved under finite closed sums and $X$ is the finite union of its closed subspaces each having $\mathfrak{P}$.
\end{proof}

We need the following lemma in the proof of Lemma \ref{ROEER}.

\begin{lemma}\label{DDA}
Let $X$ be a locally compact locally-$\mathfrak{P}$ non-$\mathfrak{P}$ space, where $\mathfrak{P}$ is a compactness-like topological property. Let $Y_i\in{\mathscr E}^*_\mathfrak{P}(X)$ for each $i\in I$, where $I$ is a non-empty index set, and let ${\mathscr F}(Y_i)=\{F_i\}$. Then $\bigvee_{i\in I} Y_i$ exists in ${\mathscr E}^*_\mathfrak{P}(X)$ and
\[\bigvee_{i\in I} Y_i=e_X\Big(\bigcap_{i\in I} F_i\Big).\]
\end{lemma}

\begin{proof}
Let
\[F=\bigcap_{i\in I} F_i.\]
Note that $F$ is a closed subspace of $\beta X\setminus X$ (and thus it is compact, as the latter is so, since $X$ is  locally compact) containing $\beta X\setminus\lambda_\mathfrak{P}X$, as each $F_i$, where $i\in I$, is closed in $\beta X$  and contains $\beta X\setminus\lambda_\mathfrak{P}X$ by Lemma \ref{Y16}. In particular, $F$ is non-empty, as $\beta X\setminus\lambda_\mathfrak{P}X$ is so by Lemma \ref{FSGK}, since $X$ is non-$\mathfrak{P}$. Let $Y=e_X(F)$. Then $Y\in{\mathscr E}^*_\mathfrak{P}(X)$. Note that ${\mathscr F}(Y)=\{F\}$. For every $i\in I$ we have $F\subseteq F_i$, and thus $Y_i\leq Y$ by Lemma \ref{DFH}. Let $Y'\in{\mathscr E}^*_\mathfrak{P}(X)$ be such that $Y_i\leq Y'$ for every $i\in I$. Let ${\mathscr F}(Y')=\{F'\}$. Then $F'\subseteq F_i$ for each $i\in I$ by Lemma \ref{DFH}, and therefore $F'\subseteq F$. Thus $Y\leq Y'$ again by Lemma \ref{DFH}.
\end{proof}

The following characterizes order-theoretically the ideal elements of ${\mathscr E}_\mathfrak{P}^*(X)$ in ${\mathscr E}^*_\mathfrak{P}(X)$.

\begin{lemma}\label{ROEER}
Let $X$ be a locally compact non-$\mathfrak{P}$ space, where $\mathfrak{P}$ is a compactness-like topological property.
\begin{itemize}
\item[\rm(1)] Let $X$ be a locally-$\mathfrak{P}$ space. Then ${\mathscr E}^*_\mathfrak{P}(X)$ has a largest element, namely
\[M^X_\mathfrak{P}=e_X(\beta X\setminus\lambda_\mathfrak{P}X).\]
\item[\rm(2)] Let $Y\in{\mathscr E}^*_\mathfrak{P}(X)$. The following are equivalent:
\begin{itemize}
\item[\rm(a)] $Y$ is an ideal.
\item[\rm(b)] If
\[Y\vee\bigvee_{i\in I} Y_i=M^X_\mathfrak{P}\]
where $Y_i\in{\mathscr E}^*_\mathfrak{P}(X)$ for each $i\in I$ and $I$ is a non-empty index set, then
\[Y\vee\bigvee_{j=1}^n Y_{i_j}=M^X_\mathfrak{P}\]
for some $i_1,\ldots,i_n\in I$.
\end{itemize}
\end{itemize}
\end{lemma}

\begin{proof}
(1). This follows from Lemmas \ref{DFH} and \ref{Y16}. Note that $\beta X\setminus\lambda_\mathfrak{P}X$ is contained in $\beta X\setminus X$, as $X\subseteq\lambda_\mathfrak{P}X$ by Lemma \ref{15}, since $X$ is locally-$\mathfrak{P}$, and $\beta X\setminus\lambda_\mathfrak{P}X$ is non-empty by Lemma \ref{FSGK}, as $X$ is non-$\mathfrak{P}$.

(2). Let ${\mathscr F}(Y)=\{F\}$. (2.a) {\em  implies} (2.b). Let
\begin{equation}\label{JJ}
Y\vee\bigvee_{i\in I}Y_i=M^X_{\mathfrak{P}},
\end{equation}
where $Y_i\in{\mathscr E}^*_\mathfrak{P}(X)$ for each $i\in I$ and $I$ is a non-empty index set. Let ${\mathscr F}(Y_i)=\{F_i\}$ for each $i\in I$. Using Lemma \ref{DDA}, it follows from (\ref{JJ}) that
\[F\cap\bigcap_{i\in I}F_i=\beta X\setminus\lambda_{\mathfrak{P}} X\]
and thus
\[\lambda_{\mathfrak{P}} X\cap F\cap\bigcap_{i\in I}F_i=\emptyset.\]
Since $\lambda_{\mathfrak{P}} X\cap F$ is compact, as $Y$ is an ideal, we have
\[\lambda_{\mathfrak{P}} X\cap F\cap\bigcap_{j=1}^kF_{i_j}=\emptyset\]
or, equivalently
\begin{equation}\label{FLL}
F\cap\bigcap_{j=1}^kF_{i_j}\subseteq\beta X\setminus\lambda_\mathfrak{P}X
\end{equation}
for some $i_1,\ldots,i_k\in I$. But by Lemma \ref{Y16} we know that $\beta X\setminus\lambda_\mathfrak{P}X$ is contained in $F$ and in each $F_{i_j}$, where $j=1,\ldots,k$. Therefore, (\ref{FLL}) yields
\[F\cap\bigcap_{j=1}^kF_{i_j}=\beta X\setminus\lambda_\mathfrak{P}X\]
and thus, again by Lemma \ref{DDA} we have
\[Y\vee\bigvee_{j=1}^kY_{i_j}=M^X_{\mathfrak{P}}.\]

(2.b) {\em  implies} (2.a). We need to show that $\lambda_{\mathfrak{P}} X\cap F$ is compact. Let $\{U_i\}_{i\in I}$  be an open cover of $\lambda_{\mathfrak{P}} X\cap F$ in $\lambda_{\mathfrak{P}}X\setminus X$. Let $i\in I$. Define
\[F_i=(\beta X\setminus X)\setminus U_i.\]
Note that $U_i$ is open in $\beta X\setminus X$, as it is open in $\lambda_{\mathfrak{P}}X\setminus X$ and the latter is open in $\beta X\setminus X$. Thus $F_i$ is closed in $\beta X\setminus X$ and therefore it is compact, as $\beta X\setminus X$ is closed in $\beta X$, since $X$ is locally compact. Also, $X\subseteq\lambda_{\mathfrak{P}}X$ by Lemma \ref{15}, as $X$ is locally-$\mathfrak{P}$ by Lemma \ref{Y16}, since ${\mathscr E}_\mathfrak{P}(X)$ is non-empty, because $I$ is so. Further,
\[(\beta X\setminus X)\setminus U_i\supseteq(\beta X\setminus X)\setminus(\lambda_{\mathfrak{P}}X\setminus X)=\beta X\setminus\lambda_{\mathfrak{P}}X,\]
and the latter in non-empty by Lemma \ref{FSGK}, as $X$ is non-$\mathfrak{P}$. That is, $F_i$ is a compact non-empty subspace of $\beta X\setminus X$ containing $\beta X\setminus\lambda_{\mathfrak{P}}X$. Define
\[Y_i=e_X(F_i).\]
Note that $Y_i\in{\mathscr E}^*_\mathfrak{P}(X)$. Now
\begin{eqnarray*}
F\cap\bigcap_{i\in I}F_i&=&F\cap\bigcap_{i\in I}\big((\beta X\setminus X)
\setminus U_i\big)\\&=&F\setminus\bigcup_{i\in I}U_i\subseteq F
\setminus(F\cap\lambda_{\mathfrak{P}} X)=F\cap(\beta X\setminus\lambda_{\mathfrak{P}} X)\subseteq\beta X\setminus\lambda_{\mathfrak{P}} X.
\end{eqnarray*}
By Lemma \ref{Y16} we know that $\beta X\setminus\lambda_\mathfrak{P}X$ is contained in $F$ and in each $F_i$, where $i\in I$. Thus, from the above, we have
\[F\cap\bigcap_{i\in I}F_i=\beta X\setminus\lambda_{\mathfrak{P}} X.\]
Therefore, by Lemma \ref{DDA} we have
\[Y\vee\bigvee_{i\in I}Y_i=M^X_{\mathfrak{P}}.\]
Using our assumption, it then follows that
\[Y\vee\bigvee_{j=1}^kY_{i_j}=M^X_{\mathfrak{P}}\]
for some $i_1,\ldots,i_k\in I$. Again, by Lemmas \ref{DFH} and \ref{DDA} we have
\[F\cap\bigcap_{j=1}^kF_{i_j}=\beta X\setminus\lambda_{\mathfrak{P}} X.\]
Thus
\[(\lambda_{\mathfrak{P}} X\cap F)\setminus\bigcup_{j=1}^kU_{i_j}=\lambda_{\mathfrak{P}} X\cap F\cap\bigcap_{j=1}^k\big((\beta X\setminus X)\setminus U_{i_j}\big)=\lambda_{\mathfrak{P}} X\cap F\cap\bigcap_{j=1}^k F_{i_j}=\emptyset\]
which implies that
\[\lambda_{\mathfrak{P}} X\cap F\subseteq \bigcup_{j=1}^kU_{i_j}.\]
\end{proof}

In the following we characterize order-theoretically the locally compact elements of ${\mathscr E}_\mathfrak{P}^*(X)$ in ${\mathscr E}^*_\mathfrak{P}(X)$.

\begin{lemma}\label{GGYH}
Let $X$ be a locally compact non-$\mathfrak{P}$ space, where $\mathfrak{P}$ is a compactness-like topological property, and let $Y\in{\mathscr E}^*_\mathfrak{P}(X)$. The following are equivalent:
\begin{itemize}
\item[\rm(1)] $Y$ is locally compact.
\item[\rm(2)] There exists an ideal element $T$ in ${\mathscr E}^*_\mathfrak{P}(X)$ satisfying the following:
\begin{itemize}
\item[\rm(a)] For every co-atom $A$ in ${\mathscr E}_\mathfrak{P}(X)$ of type $(\mathrm{I})$ either $A\geq Y$ or $A\geq T$.
\item[\rm(b)] There exists no co-atom $A$ in ${\mathscr E}_\mathfrak{P}(X)$ of type $(\mathrm{I})$ with $A\geq Y$ and $A\geq T$.
\end{itemize}
\end{itemize}
\end{lemma}

\begin{proof}
Let ${\mathscr F}(Y)=\{F\}$ and note that $\beta X\setminus\lambda_{\mathfrak{P}} X\subseteq F$ by Lemma \ref{Y16}. Also, $X\subseteq\lambda_\mathfrak{P}X$ by Lemma \ref{15}, since $X$ is locally-$\mathfrak{P}$ by Lemma \ref{Y16}, because ${\mathscr E}_\mathfrak{P}(X)$ is non-empty.

(1) {\em  implies} (2). Let
\[G=\big((\beta X\setminus X)\setminus F\big)\cup(\beta X\setminus\lambda_{\mathfrak{P}} X).\]
Note that $\beta X\setminus\lambda_\mathfrak{P}X$ is non-empty by Lemma \ref{FSGK}, as $X$ is non-$\mathfrak{P}$, and $\beta X\setminus\lambda_\mathfrak{P}X$ is contained in $\beta X\setminus X$, as $X\subseteq\lambda_\mathfrak{P}X$. Since $Y$ is locally compact, $F$ is open in $\beta X\setminus X$ by Lemma \ref{DDDS}. Therefore $(\beta X\setminus X)\setminus F$ is closed in $\beta X\setminus X$ and is thus compact, as $\beta X\setminus X$ is closed in $\beta X$, since $X$ is locally compact. Therefore, $G$ is compact, as it is the union of two compact spaces. Further, $G$ is non-empty, it is contained in $\beta X\setminus X$, and it contains $\beta X\setminus\lambda_{\mathfrak{P}}X$. Let $T=e_X(G)$. Note that
\[G\cap\lambda_{\mathfrak{P}} X=(\beta X\setminus X)\setminus F\]
is compact, and thus $T$ is an ideal. That conditions (2.a) and (2.b) hold follow from the representation given in Definition \ref{EGGH} of co-atoms in ${\mathscr E}_\mathfrak{P}(X)$ of type $(\mathrm{I})$ and the facts that
\[F\cap G=\beta X\setminus\lambda_{\mathfrak{P}} X\mbox{ and }F\cup G=\beta X\setminus X.\]

(2) {\em  implies} (1). Let ${\mathscr F}(T)=\{H\}$. Note that $\beta X\setminus\lambda_{\mathfrak{P}} X\subseteq H$ by Lemma \ref{Y16}. From conditions (2.a) and (2.b) it follows that
\[F\cap H=\beta X\setminus\lambda_{\mathfrak{P}} X\mbox{ and }F\cup H=\beta X\setminus X.\]
Therefore
\[(\beta X\setminus X)\setminus F=H\cap\lambda_{\mathfrak{P}} X\]
is compact, as $T$ is an ideal. Thus $(\beta X\setminus X)\setminus F$ is closed in $\beta X\setminus X$. Therefore, $Y$ is locally compact by Lemma \ref{DDDS}.
\end{proof}

\begin{lemma}\label{SDFG}
Let $X$ and $Y$ be locally compact non-$\mathfrak{P}$ spaces, where $\mathfrak{P}$ is a compactness-like topological property, and let
\[\Gamma:{\mathscr E}_\mathfrak{P}(X)\rightarrow{\mathscr E}_\mathfrak{P}(Y)\]
be an order-isomorphism. Suppose that $|\beta X\setminus\lambda_\mathfrak{P}X|\geq 3$ and $|\beta Y\setminus\lambda_\mathfrak{P}Y|\geq 3$. Let $T\in{\mathscr E}_\mathfrak{P}(X)$.
\begin{itemize}
\item[\rm(1)] If $T$ is a co-atom of type $(\mathrm{I})$ (of type $(\mathrm{II})$, respectively) then so is $\Gamma(T)$.
\item[\rm(2)] If $T$ is a one-point extension then so is $\Gamma(T)$.
\item[\rm(3)] If $T$ is an ideal then so is $\Gamma(T)$.
\item[\rm(4)] If $T$ is a locally compact one-point extension then so is $\Gamma(T)$.
\end{itemize}
\end{lemma}

\begin{proof}
Lemma \ref{GGO} proves (1). (Note that $X$ is locally-$\mathfrak{P}$ by Lemma \ref{Y16}, as ${\mathscr E}_\mathfrak{P}(X)$ is non-empty.) Lemma \ref{DDR} and (1) prove (2). Part (3) follows from (2) and Lemma \ref{ROEER}. Part (4) follows from (1)--(3) and Lemma \ref{GGYH}.
\end{proof}

We are now ready to give our example. For a space $X$ we write
\[X=\bigoplus_{i\in I}X_i,\]
if $X$ is the disjoint union of its open (and thus clopen) subspaces $X_i$.

\begin{example}\label{DFJ}
Let
\[X=\mathbb{N}\oplus\bigoplus_{i<\Omega}R_i\mbox{ and }Y=\bigoplus_{i<\Omega}R_i\]
where $R_i=[0,\infty)$ for each $i<\Omega$. (Here $\Omega$ is the first uncountable ordinal.) Let $\mathfrak{P}$ be the Lindel\"{o}f property.  Then
\begin{equation}\label{TPJB}
\lambda_\mathfrak{P}X=\mathrm{cl}_{\beta X}\mathbb{N}\cup\bigcup\Big\{\mathrm{cl}_{\beta X}\Big(\bigcup_{i\in J}R_i\Big):J\subseteq[0,\Omega)\mbox{ is countable}\Big\}
\end{equation}
and
\begin{equation}\label{TPO}
\lambda_\mathfrak{P}Y=\bigcup\Big\{\mathrm{cl}_{\beta Y}\Big(\bigcup_{i\in J}R_i\Big):J\subseteq[0,\Omega)\mbox{ is countable}\Big\}.
\end{equation}
To show (\ref{TPJB}), consider some $C\in\mathrm{Coz}(X)$ with Lindel\"{o}f closure $\mathrm{cl}_XC$. Then
\[\mathrm{cl}_XC\subseteq\mathbb{N}\cup\bigcup_{i\in J}R_i\]
for some countable $J\subseteq[0,\Omega)$. Therefore
\[\mathrm{cl}_{\beta X}C\subseteq\mathrm{cl}_{\beta X}\mathbb{N}\cup\mathrm{cl}_{\beta X}\Big(\bigcup_{i\in J}R_i\Big).\]
On the other hand, if $J\subseteq[0,\Omega)$ is countable, then
\[D=\mathbb{N}\cup\bigcup_{i\in J}R_i\]
is Lindel\"{o}f. Since $D$ is clopen in $X$ we have $D\in\mathrm{Coz}(X)$. Also, the closure $\mathrm{cl}_{\beta X}D$ is clopen in $\beta X$. Thus
\[\mathrm{int}_{\beta X}\mathrm{cl}_{\beta X}D=\mathrm{cl}_{\beta X}D\subseteq\lambda_\mathfrak{P}X.\]
This shows (\ref{TPJB}). A similar argument shows (\ref{TPO}). Note that $X$ contains $Y$ as a closed subspace. Since $X$ is normal, as each of its summands is so, we have $\beta Y=\mathrm{cl}_{\beta X} Y$. Thus, in particular, for any countable  $J\subseteq[0,\Omega)$ we have
\[\mathrm{cl}_{\beta Y}\Big(\bigcup_{i\in J}R_i\Big)=\mathrm{cl}_{\beta X}\Big(\bigcup_{i\in J}R_i\Big)\cap\mathrm{cl}_{\beta X} Y=\mathrm{cl}_{\beta X}\Big(\bigcup_{i\in J}R_i\Big)\]
and then, comparing (\ref{TPJB}) and (\ref{TPO}), it yields
\[\lambda_\mathfrak{P}X=\mathrm{cl}_{\beta X}\mathbb{N}\cup\lambda_\mathfrak{P}Y.\]
Note that
\[\beta X=\mathrm{cl}_{\beta X}\mathbb{N}\cup\mathrm{cl}_{\beta X}Y,\]
as $X=\mathbb{N}\cup Y$. Also, since $\mathbb{N}$ and $Y$ are disjoint zero-sets in $X$ (as they are clopen in $X$) they have disjoint closures in $\beta X$. We have
\[\beta X\setminus\lambda_\mathfrak{P}X=(\mathrm{cl}_{\beta X}\mathbb{N}\cup\mathrm{cl}_{\beta X} Y)\setminus(\mathrm{cl}_{\beta X}\mathbb{N}\cup\lambda_\mathfrak{P}Y)=\mathrm{cl}_{\beta X} Y\setminus\lambda_\mathfrak{P}Y=\beta Y\setminus\lambda_\mathfrak{P}Y.\]

In the following, in order to use Lemma \ref{SDFG}, we need to show that
\[|\beta X\setminus\lambda_\mathfrak{P}X|\geq 3\mbox{ and }|\beta Y\setminus\lambda_\mathfrak{P}Y|\geq 3.\]
We show the latter; the first may be proved analogously. Let $J\subseteq[0,\Omega)$ be uncountable. Suppose that
\[\mathrm{cl}_{\beta Y}\Big(\bigcup_{i\in J}R_i\Big)\subseteq\lambda_\mathfrak{P}Y.\]
Note that for any $K\subseteq[0,\Omega)$ the closure $\mathrm{cl}_{\beta Y}(\bigcup_{i\in K}R_i)$ is clopen in $\beta Y$, as $\bigcup_{i\in K}R_i$ is clopen in $Y$. By compactness and the representation given in (\ref{TPO}) we have
\begin{equation}\label{UUO}
\mathrm{cl}_{\beta Y}\Big(\bigcup_{i\in J}R_i\Big)\subseteq\mathrm{cl}_{\beta Y}\Big(\bigcup_{i\in {K_1}}R_i\Big)\cup\cdots\cup\mathrm{cl}_{\beta Y}\Big(\bigcup_{i\in {K_n}}R_i\Big)
\end{equation}
for some countable $K_1,\ldots,K_n\subseteq[0,\Omega)$. Intersecting both sides of (\ref{UUO}) with $Y$ it yields
\[\bigcup_{i\in J}R_i\subseteq\bigcup_{i\in {K_1}}R_i\cup\cdots\cup\bigcup_{i\in {K_n}}R_i\]
which is not true, as $J$ is uncountable, while $K_1\cup\cdots\cup K_n$ is not. Thus
\[\mathrm{cl}_{\beta Y}\Big(\bigcup_{i\in J}R_i\Big)\setminus\lambda_\mathfrak{P}Y\neq\emptyset\]
for any uncountable $J\subseteq[0,\Omega)$. Choose some pairwise disjoint uncountable $J_1,J_2,J_3\subseteq[0,\Omega)$. Note that
\[\mathrm{cl}_{\beta Y}\Big(\bigcup_{i\in {J_1}}R_i\Big)\setminus\lambda_\mathfrak{P}Y, \mathrm{cl}_{\beta Y}\Big(\bigcup_{i\in {J_2}}R_i\Big)\setminus\lambda_\mathfrak{P}Y\mbox{ and }\mathrm{cl}_{\beta Y}\Big(\bigcup_{i\in {J_3}}R_i\Big)\setminus\lambda_\mathfrak{P}Y\]
are non-empty subsets of $\beta Y\setminus\lambda_\mathfrak{P}Y$ and they are pairwise disjoint, as
\[\bigcup_{i\in {J_1}}R_i,\bigcup_{i\in {J_2}}R_i\mbox{ and }\bigcup_{i\in {J_3}}R_i\]
being pairwise disjoint clopen subspaces (and thus zero-sets) in $Y$ have disjoint closures in $\beta Y$.

Now, we show that the partially ordered sets ${\mathscr E}_\mathfrak{P}(X)$ and ${\mathscr E}_\mathfrak{P}(Y)$ are not order-isomorphic. Suppose the contrary, and let
\[\Gamma:{\mathscr E}_\mathfrak{P}(X)\rightarrow{\mathscr E}_\mathfrak{P}(Y)\]
be an order-isomorphism. Let
\[G=(\beta X\setminus X)\setminus\mathrm{cl}_{\beta X}\mathbb{N}.\]
Since $\mathbb{N}$ is clopen in $X$, its closure $\mathrm{cl}_{\beta X}\mathbb{N}$ in $\beta X$ is clopen in $\beta X$. Therefore $G$ is clopen in $\beta X\setminus X$. In particular, $G$ is compact, as $\beta X\setminus X$ is closed in $\beta X$, since $X$ is locally compact. Also, $G$ contains $\beta X\setminus\lambda_\mathfrak{P}X$ by (\ref{TPJB}). Let $S=e_X(G)$ and denote $T=\Gamma(S)$. By Lemma \ref{SDFG} then $T\in{\mathscr E}^*_\mathfrak{P}(Y)$, as $S\in{\mathscr E}^*_\mathfrak{P}(X)$ and $T$ is locally compact, as $S$ is so by Lemma \ref{DDDS}, since $G$ is open in $\beta X\setminus X$. Let ${\mathscr F}_Y(T)=\{H\}$. Note that $T$ is not the smallest element of ${\mathscr E}^*_\mathfrak{P}(Y)$, as $S$ is not the smallest element of ${\mathscr E}^*_\mathfrak{P}(X)$. (Observe that the smallest element of ${\mathscr E}^*_\mathfrak{P}(X)$ is the one-point compactification of $X$.) Thus $H$ is not the whole $\beta Y\setminus Y$. We need to show the following.

\begin{xclaim}
There exists some $i<\Omega$ such that
\[H\cap\mathrm{cl}_{\beta Y}R_i=\emptyset.\]
\end{xclaim}

\noindent \emph{Proof of the claim.}
Since $T$ is locally compact, $H$ is clopen in $\beta Y\setminus Y$ by Lemma \ref{DDDS}. Thus $H$ and $(\beta Y\setminus Y)\setminus H$ are both closed in $\beta Y\setminus Y$ and therefore in $\beta Y$, as $\beta Y\setminus Y$ is closed in $\beta Y$, since $Y$ is locally compact. By the Urysohn Lemma there exists a continuous $f:\beta Y\rightarrow[0,1]$ such that $f(t)=0$ for any $t\in(\beta Y\setminus Y)\setminus H$ and $f(s)=1$ for any $s\in H$. Let
\[V=Y\cap f^{-1}\big[[0,1/2)\big].\]
Note that
\[(\beta Y\setminus Y)\setminus H=\mathrm{cl}_{\beta Y}V\setminus Y,\]
as using Lemma \ref{LKG} we have
\[(\beta Y\setminus Y)\setminus H=f^{-1}\big[[0,1/2)\big]\setminus Y\subseteq\mathrm{cl}_{\beta Y}V\setminus Y\subseteq f^{-1}\big[[0,1/2]\big]\setminus Y=(\beta Y\setminus Y)\setminus H.\]
We have
\begin{eqnarray*}
\mathrm{bd}_YV=\mathrm{cl}_YV\cap\mathrm{cl}_Y(Y\setminus V)&\subseteq&\mathrm{cl}_{\beta Y}V\cap\mathrm{cl}_{\beta Y}(Y\setminus V)\\&\subseteq& f^{-1}\big[[0,1/2]\big]\cap f^{-1}\big[[1/2,1]\big]=f^{-1}(1/2)
\end{eqnarray*}
and thus
\[\mathrm{cl}_{\beta Y}\mathrm{bd}_YV\subseteq f^{-1}(1/2)\subseteq Y.\]
Therefore
\[\mathrm{bd}_YV=\mathrm{cl}_{\beta Y}\mathrm{bd}_YV\cap Y=\mathrm{cl}_{\beta Y}\mathrm{bd}_YV\]
is compact, as it is closed in $\beta Y$. Let
\[L=\{i<\Omega:\mathrm{bd}_YV\cap R_i\neq\emptyset\}.\]
Note that $L$ is finite, as $\mathrm{bd}_YV$ is compact (and each $R_i$, where $i<\Omega$, is open in $Y$). To prove the claim, suppose to the contrary that $\mathrm{cl}_{\beta Y}R_i\cap H$ is non-empty for each $i<\Omega$. Let $i<\Omega$. Note that
\[\mathrm{cl}_{\beta Y}R_i\setminus Y=\mathrm{cl}_{\beta Y}R_i\setminus R_i=\beta R_i\setminus R_i\]
as $\mathrm{cl}_{\beta Y}R_i$ and $\beta R_i$ are equivalent compactifications of $R_i$, because $R_i$ is closed in $Y$ (and $Y$ is normal). Since $\beta R_i\setminus R_i$ is connected (see Problem 6L of \cite{GJ}) and $H$ is clopen in $\beta Y\setminus Y$, we have
\begin{equation}\label{GFDW}
\mathrm{cl}_{\beta Y}R_i\setminus Y\subseteq H.
\end{equation}
Now, let $i<\Omega$ be such that $V\cap R_i$ is non-empty. If $\mathrm{bd}_{R_i}(V\cap R_i)$ is empty, then $V\cap R_i$ is clopen in $R_i$, and since $R_i$ is connected we have $V\cap R_i=R_i$, that is $R_i\subseteq V$. But then
\[\emptyset\neq\beta R_i\setminus R_i=\mathrm{cl}_{\beta Y}R_i\setminus Y\subseteq\mathrm{cl}_{\beta Y}V\setminus Y=(\beta Y\setminus Y)\setminus H\]
which by (\ref{GFDW}) cannot be true. Thus
\[\mathrm{bd}_YV\cap R_i=\mathrm{bd}_{R_i}(V\cap R_i)\neq\emptyset\]
that is, $i\in L$. Therefore
\[V\subseteq\bigcup_{i\in L}R_i.\]
Now, using (\ref{GFDW}), we have
\begin{eqnarray*}
(\beta Y\setminus Y)\setminus H=\mathrm{cl}_{\beta Y}V\setminus Y&\subseteq&\mathrm{cl}_{\beta Y}\Big(\bigcup_{i\in L}R_i\Big)\setminus Y\\&=&\Big(\bigcup_{i\in L}\mathrm{cl}_{\beta Y}R_i\Big)\setminus Y=\bigcup_{i\in L}(\mathrm{cl}_{\beta Y}R_i\setminus Y)\subseteq H
\end{eqnarray*}
which implies that $H=\beta Y\setminus Y$. This contradiction proves the claim.

\medskip

\noindent Fix some $i<\Omega$ such that $H\cap\mathrm{cl}_{\beta Y}R_i$ is empty and let
\[H_i=(\beta Y\setminus Y)\setminus\mathrm{cl}_{\beta Y}R_i.\]
Since $R_i$ is clopen in $Y$, its closure $\mathrm{cl}_{\beta Y}R_i$ in $\beta Y$ is clopen in $\beta Y$. Thus $H_i$ is clopen in $\beta Y\setminus Y$. In particular, $H_i$ is compact, as $\beta Y\setminus Y$ is closed in $\beta Y$, since $Y$ is locally compact. Also, $H_i$ contains $\beta Y\setminus\lambda_\mathfrak{P}Y$ by (\ref{TPO}). Let $T_i=e_Y(H_i)$. Note that $T_i$ is locally compact by Lemma \ref{DDDS}. We now prove the following; the contradiction will then complete the proof.

\begin{xclaim}
$T_i$ is a co-atom in the set of locally compact elements of ${\mathscr E}^*_\mathfrak{P}(Y)$ partially ordered by the reverse of $\leq$, while its inverse image $\Gamma^{-1}(T_i)$ is not a co-atom in the set of locally compact elements of ${\mathscr E}^*_\mathfrak{P}(X)$, partially ordered in the same way.
\end{xclaim}

\noindent \emph{Proof of the claim.} Suppose to the contrary that $\omega Y<T'< T_i$ for some locally compact $T'\in{\mathscr E}^*_\mathfrak{P}(Y)$. ($\omega Y$ denotes the one-point compactification of $Y$.) Let ${\mathscr F}_Y(T')=\{H'\}$. Then $H_i\subsetneqq H'\subsetneqq\beta Y\setminus Y$ by Lemma \ref{DFH}. Therefore
\[\emptyset\neq(\beta Y\setminus Y)\setminus H'\subsetneqq(\beta Y\setminus Y)\setminus H_i=\mathrm{cl}_{\beta Y}R_i\setminus Y=\beta R_i\setminus R_i.\]
Since $T'$ is locally compact, $H'$ (and thus $(\beta Y\setminus Y)\setminus H'$) is clopen in $\beta Y\setminus Y$ by Lemma \ref{DDDS}. This contradicts the fact that $\beta R_i\setminus R_i$ is connected. (See Problem 6L of \cite{GJ}.) Thus $T_i$ is a co-atom in the set of all locally compact elements of ${\mathscr E}^*_\mathfrak{P}(Y)$ partially ordered by the reverse of $\leq$. Let $S_i=\Gamma^{-1}(T_i)$. Note that $S_i\in{\mathscr E}^*_\mathfrak{P}(X)$ and $S_i$ is locally compact by Lemma \ref{SDFG}, as $T_i$ is so. Let ${\mathscr F}_X(S_i)=\{G_i\}$. Since $H\cap\mathrm{cl}_{\beta Y}R_i$ is empty, we have
\[H\subseteq(\beta Y\setminus Y)\setminus\mathrm{cl}_{\beta Y}R_i=H_i\]
and thus $T_i\leq T$ by Lemma \ref{DFH}. Therefore
\[S_i=\Gamma^{-1}(T_i)\leq \Gamma^{-1}(T)=S\]
and then $G\subseteq G_i$ again by Lemma \ref{DFH}. Note that $\mathrm{cl}_{\beta X}\mathbb{N}$ and $\beta\mathbb{N}$ are equivalent compactifications of $\mathbb{N}$, as $\mathbb{N}$ is closed in $X$ (and $X$ is normal). We have
\[(\beta X\setminus X)\setminus G_i\subseteq(\beta X\setminus X)\setminus G=\mathrm{cl}_{\beta X}\mathbb{N}\setminus X=\mathrm{cl}_{\beta X}\mathbb{N}\setminus \mathbb{N}=\beta\mathbb{N}\setminus\mathbb{N}.\]
Therefore $(\beta X\setminus X)\setminus G_i$ may be regarded as a clopen subspace of $\beta\mathbb{N}\setminus\mathbb{N}$ and thus homeomorphic to $\beta\mathbb{N}\setminus\mathbb{N}$ itself. (See Exercise 3.6.A of \cite{E}.) By Lemma \ref{DDDS} the complement (in $\beta X\setminus X$) of each clopen subspace of $(\beta X\setminus X)\setminus G_i$ corresponds to a locally compact element $S'$ of ${\mathscr E}^*_\mathfrak{P}(X)$ which satisfies $S'\leq S_i$ by Lemma \ref{DFH}. That is $S_i$ is not a co-atom in the set of locally compact elements of ${\mathscr E}^*_\mathfrak{P}(X)$ partially ordered by the reverse of $\leq$. This proves the claim.
\end{example}

Our concluding results determine $\lambda_\mathfrak{P}X$ in the cases when $\mathfrak{P}$ is either pseudocompactness or realcompactness. Note that neither pseudocompactness nor realcompactness is a compactness-like topological property (indeed, pseudocompactness is not inverse invariant under perfect mappings and realcompactness is not invariant under perfect mappings), so the results of our first two sections are not deducible from the results of this section.

Proposition \ref{PTF} is known (see \cite{Ko4}); it is included here for completeness of results.

The following result is due to  A.W. Hager and D.G. Johnson in \cite{HJ}; a direct proof may be found in \cite{C}. (See also Theorem 11.24 of \cite{W}.)

\begin{lemma}[Hager and Johnson \cite{HJ}]\label{A}
Let $U$ be an open subspace of a space $X$. If $\mathrm{cl}_{\upsilon X} U$ is compact then $\mathrm{cl}_X U$ is pseudocompact.
\end{lemma}

\begin{lemma}\label{HGA}
Let $A$ be a regular closed subspace of a space $X$. Then $\mathrm{cl}_{\beta X} A\subseteq\upsilon X$ if and only if $A$ is pseudocompact.
\end{lemma}

\begin{proof}
The first half follows from Lemma \ref{A}. For the second half, note that if $A$ is pseudocompact then so is $\mathrm{cl}_{\upsilon X} A$. But  $\mathrm{cl}_{\upsilon X} A$, being closed in $\upsilon X$, is also realcompact, and thus compact. Therefore $\mathrm{cl}_{\beta X} A\subseteq\mathrm{cl}_{\upsilon X} A$.
\end{proof}

\begin{proposition}\label{PTF}
Let $X$ be a space and let $\mathfrak{P}$ be pseudocompactness. Then
\[\lambda_\mathfrak{P}X=\mathrm{int}_{\beta X}\upsilon X.\]
\end{proposition}

\begin{proof}
If $C\in\mathrm{Coz}(X)$ has pseudocompact closure in $X$ then $\mathrm{cl}_{\beta X} C\subseteq\upsilon X$, by Lemma \ref{HGA}, and then
\[\mathrm{int}_{\beta X}\mathrm{cl}_{\beta X} C\subseteq\mathrm{int}_{\beta X}\upsilon X.\]

For the reverse inclusion, let $t\in\mathrm{int}_{\beta X}\upsilon X$. Let $f:\beta X\rightarrow[0,1]$ be continuous with $f(t)=0$ and $f(s)=1$ for any $s\in\beta X\setminus\mathrm{int}_{\beta X}\upsilon X$. Then
\[C=X\cap f^{-1}\big[[0,1/2)\big]\in\mathrm{Coz}(X)\]
and $t\in\mathrm{int}_{\beta X}\mathrm{cl}_{\beta X} C$, as $t\in f^{-1}[[0,1/2)]$ and
\[f^{-1}\big[[0,1/2)\big]\subseteq\mathrm{int}_{\beta X}\mathrm{cl}_{\beta X} C\]
by Lemma \ref{LKG}. Also, $\mathrm{cl}_X C$ is pseudocompact, by Lemma \ref{HGA}, as
\[\mathrm{cl}_{\beta X}C\subseteq f^{-1}\big[[0,1/2]\big]\subseteq\upsilon X.\]
\end{proof}

\begin{lemma}[Gillman and Jerison \cite{GJ}]\label{DDJD}
If $A$ is a $C$-embedded subspace of a space $X$ then $\mathrm{cl}_{\upsilon X}A=\upsilon A$.
\end{lemma}

Note that in a normal space each closed subspace is $C$-embedded. (See Theorem 1.10 (g) of \cite{PW}.)

\begin{proposition}\label{RRGY}
Let $X$ be a normal space and let $\mathfrak{P}$ be realcompactness. Then
\[\lambda_\mathfrak{P}X=\beta X\setminus\mathrm{cl}_{\beta X}(\upsilon X\setminus X).\]
\end{proposition}

\begin{proof}
Let $C\in\mathrm{Coz}(X)$ have realcompact closure in $X$. Then since $\mathrm{cl}_X C$ is $C$-embedded in $X$, as $X$ is normal, by Lemma \ref{DDJD} we have \[\mathrm{cl}_{\upsilon X}C=\upsilon(\mathrm{cl}_X C)=\mathrm{cl}_X C.\]
But then
\[\mathrm{int}_{\beta X}\mathrm{cl}_{\beta X} C\cap(\upsilon X\setminus X)=\emptyset,\]
as
\[\mathrm{cl}_{\beta X} C\cap(\upsilon X\setminus X)=\mathrm{cl}_{\upsilon X} C\cap(\upsilon X\setminus X)=\emptyset\]
and thus
\[\mathrm{int}_{\beta X}\mathrm{cl}_{\beta X} C\cap\mathrm{cl}_{\beta X}(\upsilon X\setminus X)=\emptyset,\]
or, equivalently
\[\mathrm{int}_{\beta X}\mathrm{cl}_{\beta X} C\subseteq\beta X\setminus\mathrm{cl}_{\beta X}(\upsilon X\setminus X).\]

For the reverse inclusion, let $t\in\beta X\setminus\mathrm{cl}_{\beta X}(\upsilon X\setminus X)$. Let $f:\beta X\rightarrow[0,1]$ be continuous with $f(t)=0$ and $f(s)=1$ for any $s\in\mathrm{cl}_{\beta X}(\upsilon X\setminus X)$. Then
\[C=X\cap f^{-1}\big[[0,1/2)\big]\in\mathrm{Coz}(X)\]
and $t\in\mathrm{int}_{\beta X}\mathrm{cl}_{\beta X} C$, as $t\in f^{-1}[[0,1/2)]$ and
\[f^{-1}\big[[0,1/2)\big]\subseteq\mathrm{int}_{\beta X}\mathrm{cl}_{\beta X} C\]
by Lemma \ref{LKG}. Also, since $\mathrm{cl}_{\beta X}C\cap(\upsilon X\setminus X)$ is empty, the closure
\[\mathrm{cl}_X C=X\cap\mathrm{cl}_{\beta X}C=\upsilon X\cap\mathrm{cl}_{\beta X}C=\mathrm{cl}_{\upsilon X}C,\]
being closed in $\upsilon X$, is realcompact.
\end{proof}

\section*{Acknowledgement} The author wishes to thank the anonymous referee for reading the manuscript and comments which helped improving the exposition of the article.


\begin{thebibliography}{10}

\bibitem[1]{B} G. Birkhoff, Lattice Theory, American Mathematical Society, New York, 1948.

\bibitem[2]{Bu} D.K. Burke, Covering properties, in: K. Kunen and J.E. Vaughan (Eds.), Handbook of Set-theoretic Topology, Elsevier, Amsterdam, 1984, pp. 347--422.

\bibitem[3]{C} W.W. Comfort, On the Hewitt realcompactification of a product space, Trans. Amer. Math. Soc., {\bf 131} (1968), 107--118.

\bibitem[4]{Do} J.M. Dom\'{\i}nguez, A generating family for the Freudenthal compactification of a class of rimcompact spaces, Fund. Math., {\bf 178} (2003), 203--215.

\bibitem[5]{E} R. Engelking, General Topology, Second edition, Heldermann Verlag, Berlin, 1989.

\bibitem[6]{GJ} L. Gillman and M. Jerison, Rings of Continuous Functions, Springer--Verlag, New York--Heidelberg, 1976.

\bibitem[7]{HJ} A.W. Hager and D.G. Johnson, A note on certain subalgebras of $C(X)$, Canad. J. Math., {\bf 20} (1968), 389--393.

\bibitem[8]{Ko1} M.R. Koushesh, On one-point metrizable extensions of locally compact metrizable spaces, Topology Appl., {\bf 154} (2007), 698--721.

\bibitem[9]{Ko2} M.R. Koushesh, On order-structure of the set of one-point Tychonoff extensions of a locally compact space, Topology Appl., {\bf 154} (2007), 2607--2634.

\bibitem[10]{Ko3} M.R. Koushesh, Compactification-like extensions, Dissertationes Math. (Rozprawy Mat.), {\bf 476} (2011), 88 pp.

\bibitem[11]{Ko4} M.R. Koushesh, The partially ordered set of one-point extensions, Topology Appl., {\bf 158} (2011), 509--532.

\bibitem[12]{Ko6} M.R. Koushesh, One-point extensions of locally compact paracompact spaces, Bull. Iranian Math. Soc., {\bf 37} (2011), no. 4, 199--228.

\bibitem[13]{Ko5} M.R. Koushesh, A pseudocompactification, Topology Appl., {\bf 158} (2011), 2191--2197.

\bibitem[14]{Ko9} M.R. Koushesh, The Banach algebra of continuous bounded functions with separable support, Studia Math., {\bf 210} (2012), 227--237.

\bibitem[15]{Ko8} M.R. Koushesh, Connectedness modulo a topological property, Topology Appl., {\bf 159} (2012), 3417--3425.

\bibitem[16]{Ko7} M.R. Koushesh, One-point extensions and local topological properties, Bull. Aust. Math. Soc., {\bf 88} (2013), 12--16.

\bibitem[17]{MRW1} J. Mack, M. Rayburn and R.G. Woods, Lattices of topological extensions, Trans. Amer. Math. Soc., {\bf 189} (1974), 163--174.

\bibitem[18]{Mag3} K.D. Magill, Jr., The lattice of compactifications of a locally compact space, Proc. London Math. Soc., {\bf 18} (1968), 231--244.

\bibitem[19]{Me} F. Mendivil, Function algebras and the lattice of compactifications, Proc. Amer. Math. Soc., {\bf 127} (1999), 1863--1871.

\bibitem[20]{Mr} S. Mr\'{o}wka, On local topological properties, Bull. Acad. Polon. Sci., {\bf 5} (1957), 951--956.

\bibitem[21]{PW} J.R. Porter and R.G. Woods, Extensions and Absolutes of Hausdorff Spaces, Springer--Verlag, New York, 1988.

\bibitem[22]{PW1} J.R. Porter and R.G. Woods, The poset of perfect irreducible images of a space, Canad. J. Math., {\bf 41} (1989), 213--233.

\bibitem[23]{R} M.C. Rayburn, On Hausdorff compactifications, Pacific J. Math., {\bf 44} (1973), 707--714.

\bibitem[24]{Steph} R.M. Stephenson, Jr., Initially $\kappa$-compact and related spaces, in: K. Kunen and J.E. Vaughan (Eds.), Handbook of Set-theoretic Topology, Elsevier, Amsterdam, 1984, pp. 603--632.

\bibitem[25]{Va} J.E. Vaughan, Countably compact and sequentially compact spaces, in: K. Kunen and J.E. Vaughan (Eds.), Handbook of Set-theoretic Topology, Elsevier, Amsterdam, 1984, pp. 569--602.

\bibitem[26]{W} M.D. Weir, Hewitt--Nachbin Spaces, American Elsevier, New York, 1975.

\bibitem[27]{Wo1} R.G. Woods, Zero-dimensional compactifications of locally compact spaces, Canad. J. Math., {\bf 26} (1974), 920--930.

\end{thebibliography}
\end{document}